\documentclass[11pt]{a_a}

\usepackage{amssymb}
\usepackage [russian,english]{babel}

\theoremstyle{plain}

\newtheorem{lemma}{Lemma}

\theoremstyle{definition}
\newtheorem{definition}{Definition}

\newtheorem{remark}{Remark}

\theoremstyle{plain}
\newtoks\thehProclaim
\newtheorem*{Proclaim}{\the\thehProclaim}
\newenvironment{proclaim}[1]{\thehProclaim{#1}\begin{Proclaim}}{\end{Proclaim}}

\theoremstyle{definition}
\newtoks{\thehRemark}
\newtheorem*{Remark}{\the\thehRemark}

\renewcommand{\leq}{\leqslant}
\renewcommand{\geq}{\geqslant}

\begin{document}

\title[$C^1$-nondensity of OSP]{Nondensity of orbital shadowing property in $C^1$-topology.}

\author{Osipov A.V.}

\begin{abstract}
The orbital shadowing property (OSP) of discrete dynamical systems on smooth closed manifolds is considered. Nondensity of OSP with respect to the $C^1$-topology is proved. The proof uses the method of skew products developed by Yu.S. Ilyashenko and A.S. Gorodetski.
\end{abstract}

\maketitle

\section{Introduction}

The theory of shadowing studies the problem of closeness of approximate and exact trajectories (or orbits) of dynamical systems on unbounded time intervals. This problem is important both for applications (as a rule, approximate trajectories generated by computer simulation of a system are considered) and for the qualitative theory of dynamical systems (shadowing properties can be considered as weak forms of structural stability). Note that we consider only discrete-time dynamical systems (cascades) generated by homeomorphisms of metric spaces and diffeomorphisms of closed smooth manifolds. In this paper, we do not distinguish between a homeomorphism and the dynamical system generated by this homeomorphism. Roughly speaking, a cascade has one of the shadowing properties if any ''sufficiently precise'' approximate trajectory is ''close'' to an exact one. Since the statement that approximate trajectories (pseudotrajectories) and exact trajectories are close can be formalized in various ways, there are several shadowing properties. Let us mention the pseudo orbit tracing property POTP, the orbital shadowing property OSP, and the weak shadowing property WSP. The state of the art of the theory of shadowing is described in the monographs~\cite{8,9}. Let us give exact definitions of the shadowing properties that are used in this paper.

Let $f$ be a homeomorphism of a metric space $M$ with metric dist. Let us recall the definitions of the exact trajectory of a point $p\in M$ of the homeomorphism $f$ and its positive and negative semitrajectories:
$$O(p,f)=\left\{f^k(p)\mid k\in\textbf{Z}\right\},$$
$$O_{+}(p,f)=\left\{f^k(p)\mid k\in\textbf{Z}, k\geq 0\right\},$$
$$O_{-}(p,f)=\left\{f^k(p)\mid k\in\textbf{Z}, k\leq 0\right\}$$
(hereinafter, we denote by $\textbf{Z}$ the set of integers).

For convenience, without additional remarks, we often use the notation
$$p_k=f^k(p)\quad\textrm{for }k\in\textbf{Z}.$$
In addition, we sometimes identify a periodic point with its trajectory, i.e. the set $O(p,f)$.

As usual, we say that a sequence $\xi=\{x_k\}\subset M$ is a $d$-pseudotrajectory if 
$$\textrm{dist}(x_{k+1},f(x_k))<d\qquad\textrm{for}\ k\in\textbf{Z}.$$
Thus, a $d$-pseudotrajectory is one of possible formalizations of the notion of an approximate trajectory.

We say that the homeomorphism $f$ of the space $M$ has POTP (pseudo orbit tracing property) if for any $\epsilon>0$ there exists a $d$ such that for any $d$-pseudotrajectory $\xi=\{x_k\}$ one can find a point $q\in M$ such that
$$\textrm{dist}(x_k,f^k(q))< \epsilon\quad\textrm{ for }k\in\textbf{Z}.$$
In other words, POTP means that any ''sufficiently precise'' approximate trajectory is shadowed by an exact trajectory (i.e. is pointwise close to it).

By $N(\epsilon,A)$ denote the $\epsilon$-neighborhood of a set $A\subset  M$. In the paper \cite{7}, definitions of the orbital shadowing property (OSP) and the weak shadowing property (WSP) are given. We say that the homeomorphism $f$ of the space $M$ has OSP and write $f\nobreak\in\nobreak\textrm{OSP}$ if for any $\epsilon>0$ there exists a $d>0$ such that for any $d$-pseudotrajectory $\xi$ one can find a point $q\in M$ such that 
\begin{equation}
\label{1.1}
\xi\subset N(\epsilon,O(q,f))\qquad\textrm{and}\qquad O(q,f)\subset N(\epsilon,\xi).
\tag{1.1}
\end{equation}
We say that the homeomorphism $f$ of the space $M$ has WSP if for any $\epsilon>0$ there exists a $d>0$ such that for any $d$-pseudotrajectory $\xi$ one can find a point $q\in M$ such that 
$$\xi\subset N(\epsilon,O(q,f)).$$

OSP is a weak analog of POTP: the difference is that we do not require a point $x_k$ of a pseudotrajectory $\xi=\{x_k\}$ and the point $f^k(q)$ of an exact trajectory $O(q,f)$ to be close ''at any time moment'', instead, the sets of the points of the pseudotrajectory $\xi$ and the trajectory $O(q,f)$ are required to be close. The weak shadowing property WSP is a weak variant of OSP: the difference is that a set of points of a ''sufficiently precise'' pseudotrajectory $\xi$ is required to be contained in a small neighborhood of some exact trajectory $O(q,f)$.

Let $M$ be a closed smooth manifold. As usual, denote by $\textrm{Diff}^1(M)$ the set of diffeomorphisms of the manifold $M$ with the $C^1$-topology (cf., e.g., \cite{16} for definition).

Generic properties are objects of a special interest in the theory of dynamical systems. We say that a property is generic if it holds for all cascades from a Baire second category set (cf., e.g., \cite{16} for definition) in a space of dynamical systems with a certain topology, and we say that a property is dense if it holds for all cascades from a dense set. In the paper \cite{14}, S. Yu. Pilyugin and O. B. Plamenevskaya proved the genericity of POTP with respect to the $C^0$-topology if the phase space is a closed smooth manifold (hereinafter, we consider, precisely, this case). The genericity of POTP with respect to the $C^0$-topology implies the $C^0$-genericity of OSP and WSP. Ch. Bonatti, L. J. Diaz and G. Turcat \cite{13} proved that POTP is nondense with respect to the $C^1$-topology, and S. Crovisier \cite{15} proved that WSP is $C^1$-dense (also, cf., the paper of S. Yu. Pilyugin, K. Sakai and O. A. Tarakanov \cite{12}).
      
Our main goal is to prove the $C^1$-nondensity (and, therefore, the $C^1$-nongene\-ricity) of OSP, which takes the ''intermediate'' position between WSP and POTP. As usual, denote by $S^2$ the two-dimensional sphere and by $S^1$ the circle. Our main result is the following theorem:

\begin{proclaim}{Theorem A} 
There exists a domain $W\subset\textrm{Diff}^1(S^2\times S^1)$ such that any diffeomorphism $f\in W$ does not have OSP.   
\end{proclaim}

In order to prove it, we use an idea originating in works of A. S. Gorodetski and Yu. S. Ilyashenko: to construct the example in a class of partially hyperbolic skew products. To be precise, we consider a certain step skew product $G_0$ over the Bernoulli shift $\sigma$ with the fibre homeomorphic to the circle (all necessary definitions are given later). Having realized the Bernoulli shift as a mapping of the Smale horseshoe that is sufficiently fast contracting and expanding compared with the fibre dynamics, we see that a local maximal partially hyperbolic set  with center fibres homeomorphic to the circle corresponds to this skew product. Furthermore, the technique of Hirsch-Pugh-Shub-Gorodetski (cf.~\cite{5,6,11}) implies that the partially hyperbolic set persists under small perturbations of this smooth realization remaining the product of the circle and the Smale horseshoe. And, due to Holder dependence of (individually smooth) center fibres on the point in the base (i.e., in the Smale horseshoe), this skew product is a Holder mild skew product. 

We take a sufficiently small $C^1$-neighborhood of the step skew product $G_0$ as the required neighborhood $W$ from Theorem A. In particular, the neighborhood is chosen so small that any diffeomorphism from $W$ is assigned to some mild skew product. Further, we show that Theorem A can be reduced to Theorem $\textrm{A}^{\prime}$ (which will be exactly formulated in the next section).

\begin{proclaim}{Theorem $\textrm{A}^{\prime}$}
Any Holder mild skew product ''sufficiently close'' to the skew product $G_0$ does not have OSP.
\end{proclaim}  

The proof of Theorem $\textrm{A}^{\prime}$ is split into two cases. The first case (Case (A1)) corresponds to the situation when there exist two hyperbolic periodic points $p$ (with the one-dimensional unstable manifold) and $q$ (with the one-dimensional stable manifold) such that this manifolds intersect. In this case, using Main Lemma, we construct a pseudotrajectory that can not be orbitally shadowed by any exact trajectory. 

The second case (Case (A2)) corresponds to the situation when there are no such intersections. In this case, we construct a pseudotrajectory such that any exact trajectory that orbitally shadows the pseudotrajectory should be the heteroclinic trajectory connecting two hyperbolic periodic points with the one-dimensional unstable and the one-dimensional stable manifolds, respectively. The assumption that the cascade has OSP contradicts to the assumption about the absense of such intersections.

Let us describe the further structure of the paper. In Sec. 2, main definitions are given, main properties of skew products are described, and it is shown that Theorem A can be reduced to Theorem $\textrm{A}^{\prime}$. In Sec. 3, Lemma~1 (Main Lemma), which plays a significant role in the proof of Theorem $\textrm{A}^{\prime}$, is formulated and proved. In Sec. 4, it is shown that the proof of Theorem~$\textrm{A}^{\prime}$ can be reduced to consideration of two cases: Case (A1) and Case (A2), and Case (A1) is proved. In addition, a scheme of the proof of Theorem $\textrm{A}^{\prime}$ is briefly outlined at the beginning of Sec. 4. In Sec. 5, two auxiliary lemmas on properties of the skew products under consideration, which are necessary for the proof of Case (A2), are formulated and proved. In Sec. 6, Case (A2) is proved with an exception of Lemma 6. Lemma 6 is proved in Sec. 7, which consists of 4 subsections. In Subsec. 7.1, main notions required for the proof of Lemma 6 are inroduced. In Subsec. 7.2, outlines of proofs of Lemmas 8 and 9 playing a key role in the proof of item (6.c) of Lemma 6 are given. In Subsec. 7.3, the proof of item (6.c) is completed; and, finally, in Subsec. 7.4, the remaining items of Lemma 6 are proved.     


\section{Dynamical properties of skew products}

Let us give main definitions.

By $\Sigma^2$ denote the space of all two-sided sequences of 0 and 1 with the metric
$$d_{\Sigma^2}(\omega,\omega^{\prime})=1/2^{k},$$
where $k\geq 0$ is the minimal integer number such that if $\omega=\ldots\beta_{-1}|\beta_0\beta_1\ldots$ and $\omega^{\prime}=\ldots\beta_{-1}^{\prime}|\beta_0^{\prime}\beta_1^{\prime}\ldots$,
then 
$$\beta_{-k-1}\neq\beta_{-k-1}^{\prime}\qquad\textrm{ or }\qquad\beta_{k}\neq\beta_{k}^{\prime},$$
and the sign $|$ means that the next symbol stands at the zero position. The sign $|$ is used further in the paper. Let us recall the definition of the Bernoulli shift $\sigma\nobreak:\nobreak\Sigma^2\mapsto\Sigma^2$:
$$\sigma(\ldots\beta_{-1}|\beta_0\beta_1\beta_2\ldots)=\ldots\beta_{-1}\beta_0|\beta_1\beta_2\ldots.$$

In the paper~\cite{1}, the following definitions are introduced:

\begin{definition}
Fix two diffeomorphisms $f_0$ and $f_1$ of the circle $S^1$. A step skew product is a mapping $G:\Sigma^2\times S^1\mapsto\Sigma^2\times S^1$ such that
$$G(\omega,\phi)=(\sigma(\omega), f_{\omega_0}(\phi))\qquad\textrm{for}\ \omega\in\Sigma^2,\phi\in S^1,$$
where $\omega_0$ is the symbol standing at the zero position of the sequence $\omega$. 
\end{definition}  

\begin{definition}
Fix a family of diffeomorphisms $f_{\omega}$ of the circle $S^1$ that is parameterized by two-sided sequences $\omega\in\Sigma^2$. A mild skew product is a mapping $G:\Sigma^2\times S^1\mapsto\Sigma^2\times S^1$ such that
$$G(\omega,\phi)=(\sigma(\omega), f_{\omega}(\phi))\qquad\textrm{for}\ \omega\in\Sigma^2,\phi\in S^1.$$
\end{definition}

Let us emphasize that in Definition 1, the choice of the diffeomorphism $f_j$ is completely determined by the symbol $\omega_0$ standing at the zero position of the sequence $\omega$, whereas in Definition 2, the choice depends on the whole sequence $\omega$. By $g$ denote a diffeomorphism of the sphere $S^2$ that has a standard Smale horseshoe. It is well known that the mapping $g$ has a locally maximal invariant subset $\Lambda$ homeomorphic to the set $\Sigma^2$, and that the restriction of the mapping $g$ to the set $\Lambda$ is topologically conjugate with the Bernoulli shift $\sigma$ (cf., e.g.,~\cite{2}).

It is well known (cf., e.g.,~\cite{1}) that the diffeomorphism $g:S^2\mapsto S^2$ can be considered as a mapping $g:D_0\cup D_1\mapsto D^\prime_0\cup D^\prime_1$, where $D_0$ and $D_1$ are disjoint horizontal rectangles, and $D_0^{\prime}$ and $D_1^{\prime}$ are disjoint vertical rectangles. In the following definition (which is also taken from the paper~\cite{1}), we extend a step skew product to the set $(D_0\cup D_1)\times S^1$:

\begin{definition} 
The smooth realization of a step skew product $G$ is a smooth mapping $F:(D_0\cup D_1)\times S^1\mapsto\allowbreak(D^{\prime}_0\nobreak\cup\nobreak D^{\prime}_1)\nobreak\times\nobreak S^1$ such that
$$F(x,\phi)=(g(x),f_{x}(\phi))\qquad\textrm{for}\ x\in D_0\cup D_1,\ \phi\in S^1,$$
$$\textrm{where}\ f_{x}=f_j\qquad\textrm{for}\ x\in D_j,\ j\in\{0,1\},$$
and $f_0$ and $f_1$ are the diffeomorphisms from the definition of the step skew product $G$.
\end{definition}

The smooth realization $F$ of a skew product $G$ can be smoothly extended to a diffeomorphism of the manifold $M=S^2\times S^1$. We denote this extension by $F$ again, and, hereinafter, we understand by a smooth realization precisely a diffeomorphism of the manifold $M$. It is easily seen that the diffeomorphism $F$ has a locally maximal invariant set such that the dynamics on this set coincides with the dynamics of the initial skew product $G$. 

Let $g_0$ be the rotation of the circle $S^1$ by small angle $b<1/100$. Let $g_1$ be an orientation preserving diffeomorphism whose non-wandering set consists only of two fixed points: an attractor $p$ and a repeller $q$. As usual, we consider $S^1$ as \textbf{R}/\textbf{Z}. We assume that the mapping $g_1$ is the linear expansion with some constant $a>1$ in the neighborhood of the point $q=0$ of radius $1/8$  and the linear contraction with the constant $1/a$ in the neighborhood of the point $p=1/2$ of radius $1/8$. As usual, we denote by id the identity map. In addition, we assume that the diffeomorphisms $g_0$ and $g_1$ satisfy the formula
$$\textrm{dist}_{C^1}(g_j,\textrm{id})<\delta_0\quad\textrm{for}\ j\in\{0,1\},$$
where $\delta_0$ is a sufficiently small number (we will impose restrictions on the size of $\delta_0$ further, in Theorem 2). By $G_0$ denote the step skew product generated by the diffeomorphisms $g_0$ and $g_1$. $G_0$ is precisely the skew product discussed in Sec. 1.

The set $\Sigma^2$ is called the base, and any set of the form $\omega\times S^1$, where $\omega\in\Sigma^2$, is called a fibre. 
We denote by $pr:\Sigma^2\times S^1\mapsto \Sigma^2$ the natural projection onto the base. We say that the trajectories of points $p_1,p_2\in\Sigma^2\times S^1$ lie on different fibres if the trajectories of their projections to the base do not intersect. Any finite sequence of 0 or 1 is called a word.

A definition of a hyperbolic periodic point is given in the book \cite{16}. Let $p$ be a hyperbolic periodic point of a diffeomorphism $f$ of a manifold $M$. Let us define the sets
\begin{equation}
\label{2.1}
W^s(p)=\left\{q\in M | \exists r\in O(p,f): \textrm{dist}\left(f^{k}(q),f^k(r)\right)\longrightarrow 0,\ k\rightarrow+\infty\right\},
\tag{2.1}
\end{equation}
\begin{equation}
\label{2.2}
W^u(p)=\left\{q\in M | \exists r\in O(p,f): 
\textrm{dist}\left(f^{k}(q),f^k(r)\right)\longrightarrow 0,\ k\rightarrow-\infty\right\}.
\tag{2.2}
\end{equation}
For convenience, we call the sets defined by equalities $(\ref{2.1})$ and $(\ref{2.2})$ the stable and the unstable manifolds of the point $p$, respectively. Note that usually this sets are called the stable and unstable manifolds of the trajectory $O(p,f)$. We say that the point $p$ is a point of type $(m,n)$ if $$\textrm{dim}W^s(p)=m\quad\textrm{and}\quad\textrm{dim}W^u(p)=n.$$

A periodic point $p\in\Sigma^2$ of the Bernoulli shift $\sigma$ is called a hyperbolic periodic point of type (1,1). According to this definition, a hyperbolic periodic point of type (1,1) of the diffeomorphism $g\nobreak:\nobreak S^2\mapsto S^2$ corresponds to any hyperbolic periodic point $p\in\Sigma^2$. A periodic point $p=(\omega_0,\phi_0)\in\Sigma^2\times S^1$ is called a hyperbolic periodic point of type (2,1) (of type (1,2), respectively) of a mild skew product $G$ if it is a hyperbolic attracting (repelling) point of the diffeomorphism
$$\bar{G}_{\omega_0,\ \phi_0}:S^1\mapsto S^1,\quad \bar{G}_{\omega_0,\ \phi_0}(\phi):=pr_{S^1}G^{m_p}(\omega_0,\phi)\quad\textrm{for }\phi\in S^1,$$   
where $pr_{S^1}$ is a projection onto $S^1$, and $m_p$ is the period of the point $p$.

We use a result of Ilyashenko and Gorodetski on density of hyperbolic periodic points of different types, which is, in fact, a consequence of Theorem 2 in \cite{4}: 

\begin{proclaim}{Theorem 1 (Gorodetski, Ilyashenko)} For the diffeomorphisms $g_0$ and $g_1$ defined above and any numbers $C$ and $\alpha$ there exist neighborhoods $W_0(g_0)$ and $W_1(g_1)$ (in the $C^1$-topology) such that if a mild skew product $G$ (generated by diffeomorphisms $f_{\omega}$) satisfies the conditions:
\begin{equation}
\label{2.3}
f_{\omega}\in W_{\omega_0}\qquad\textrm{for}\ \omega\in\Sigma^2
\tag{2.3}
\end{equation}
(where $\omega_0$ is the symbol standing at the zero position of the sequence $\omega$);
\begin{equation}
\label{2.4}
L:=\max_{\omega\in\Sigma^2}\max_{\phi\in S^1}(||Df_{\omega}(\phi)||,||Df^{-1}_\omega(\phi)||)<2^{\alpha};
\tag{2.4}
\end{equation}
\begin{equation}
\label{2.5}
d_{C^0}(f_{\omega},f_{\omega^\prime})\leq C(d_{\Sigma^2}(\omega,\omega^\prime))^{\alpha}\qquad\textrm{for}\ \omega,\omega^\prime\in\Sigma^2,
\tag{2.5}
\end{equation}
where $d_{C^0}$ is the $C^0$-metric, then both hyperbolic periodic points of type (2,1) and hyperbolic periodic points of type (1,2) are dense in the set $\Sigma^2\times S^1$.
\end{proclaim}

Note that if diffeomorphisms $f_{\omega}$ satisfy relation $(\ref{2.5})$, then the mild skew product $G$ is called a Holder mild skew product. In fact, Theorem 1 states that hyperbolic periodic points of different types are dense for Holder mild skew products ''sufficiently close'' to the skew product~$G_0$. 

By $F_0$ denote the smooth realization of the step skew product $G_0$. The following theorem plays an important role in the proof of Theorem A. In fact, the theorem states that any diffeomorphism close to the diffeomorphism $F_0$    
has a local maximal invariant set such that the dynamics on this set coincides with the dynamics of some Holder skew product ''close'' to the skew product~$G_0$. 

\begin{proclaim}{Theorem 2}
Suppose that the diffeomorphisms $g_0$ and $g_1$ defined above are sufficiently $C^1$-close to the identity diffeomorphism (i.e, the number $\delta_0$ defined above is sufficiently small). Then there exist numbers $C$ and $\alpha$, and a neighborhood $W$ of the diffeomoprhism $F_0$ in the $C^1$-topology such that any diffeomorphism $F\in W$ has a local maximal invariant set $\Delta$, and $F|_{\Delta}$ is topologically conjugate with a mild skew product $G$, which satisfies all conditions of Theorem 1 (namely, conditions $(\ref{2.3})$, $(\ref{2.4})$ and $(\ref{2.5})$). 
\end{proclaim}

\begin{remark}
It seems that Theorem 2 has not been formulated anywhere, but in fact, its proof is given in the papers \cite{3,5,11}. In Theorem 2, a mild skew product $G$ is non-one-to-one assigned to a diffeomorphism $F$, which is sufficiently close to the identity. In the papers \cite{3,5,11}, it is proved that this correspondence is continuous. Note that the reasoning of the papers \cite{5,11}  implies that hyperbolic periodic points from $\Delta$ of diffeomorphisms from $W$ are assigned to hyperbolic periodic points of the same type of mild skew products. 
\end{remark} 

In order to prove the main result, we need the following notations.

For convenience, we denote by dist both the metric on the manifold $S^2\times S^1$ and the metric in the space $\Sigma^2\times S^1$. Let $p$ be a hyperbolic periodic point of a homeomorphism $f$ of a metric space $M$. The set defined by equality~$(\ref{2.1})$ (equality $(\ref{2.2})$, respectively)
is called the stable (respectively, the unstable) manifold of the point $p$. Choose a mild skew product $G$ from Theorem 2 as a homeomorphism $f$. Let $F$ be the diffeomorphism of the manifold $M$ that corresponds to the mild skew product $G$. Let us emphasize, that the sets $W^s(p)$ and $W^u(p)$ are not manifolds in spite of their names. However, the sets $W^s_F(\bar{p})$ and $W^u_F(\bar{p})$, the stable and unstable manifolds of the trajectory of the point $\bar{p}\in S^2\times S^1$ (which corresponds to the point $p\in \Sigma^2\times S^1$ in the sense of Theorem 2) with respect to the diffeomorphism $F$, are the manifolds. Define dimensions of the sets $W^s(p)$ and $W^u(p)$ by the formula
$$\textrm{dim}W^s(p):=\textrm{dim}W^s_F(\bar{p}),\qquad\textrm{dim}W^u(p):=\textrm{dim}W^u_F(\bar{p}).$$
Thus, if $p$ is a hyperbolic periodic point of type $(m,n)$, then $\textrm{dim}W^s(p)=m$ and $\textrm{dim}W^u(p)=n$.    
        
If $p$ is a periodic point of the Bernoulli shift $\sigma$, then the numbers $\textrm{dim}W^s(p)$ and $\textrm{dim}W^u(p)$ can be defined in the analogous way with the exception that a diffeomorphism $F$ should be changed by the diffeomorphism $g$ defined above. According to this definition, $\textrm{dim}W^s(p)=\textrm{dim}W^u(p)=1$, i.e, the point $p\in\Sigma^2$ is a hyperbolic periodic point of type (1,1).

Let $W$ be the neighborhood of the mapping $F_0$ from Theorem 2. In order to prove Theorem 2, it is sufficient to find a number $\delta^{\prime}$ such that $N(\delta^{\prime},F_0)\subset W$, and any diffeomorphism $F\in N(\delta^{\prime},F_0)$ does not have OSP. Let $\delta$ be an arbitrary number. By Remark 1, if $\delta^{\prime}$ is sufficiently small, then any diffeomorphism $F\in N(\delta^{\prime},F_0)$ is assigned to some mild skew product $G$, and the diffeomorphisms $f_{\omega}$ of the mild skew product $G$ (cf., Definition 2) are contained in the neighborhood $N(\delta,g_{\omega_0})$, where the symbol $\omega_0$ was defined in the conditions of Theorem~1. During the proof of Theorem A, we fix a sufficiently small number~$\delta$, i.e., we fix the number $\delta^{\prime}$ too, implicitly.  By Theorem 2, any diffeomorphism $F\in W$ has the local maximal invariant set $\Delta$. Hence, in order to prove Theorem A, it is enough to establish just that the restriction $F|_{\Delta}$  does not have OSP. Since OSP is preserved under conjugacy, Theorem A is reduced to the following theorem:

\begin{proclaim}{Theorem $\textbf{A}^{\prime}$}
Let $G$ be a mild skew product correspondent (in the sense of Remark 1) to some diffeomorphism $F$ of the manifold $M=S^2\times S^1$. In addition, suppose that the skew product $G$ satisfies the conditions of Theorem 1, and the neighborhoods $W_0(g_0)$ and $W_1(g_1)$ are sufficiently small. Then $G$ does not have OSP.
\end{proclaim}  

\section{Main lemma.}

The following lemma is a main ingredient of our constructions. It gives a sufficient condition for a pseudotrajectory to satisfy the following: any exact trajectory that orbitally shadows the pseudotrajectory is contained in the stable (or the unstable) manifold of a hyperbolic fixed point. 

\begin{lemma}
Let $M$ be a closed smooth manifold with metric dist, $f$ be a diffeomorphism of the manifold $M$, $p$ be a hyperbolic periodic point, and $q^1\in W^s(p)$. Fix numbers $R>0$ and $0<\epsilon_0<R/2$. There exists a number $0<\epsilon<\epsilon_0$ such that if a sequence $\xi=\{x_k\}$ satisfies the following relations:
\begin{equation}
\label{3.1}
x_k=q^1_k=f^k(q^1)\qquad\textrm{for}\ k\geq1,\quad\ x_k\notin N(\epsilon_0,O(p,f))\qquad\textrm{for}\ k<1,
\tag{3.1}
\end{equation}
then for any point $q^2$ which satisfies the inclusions 
\begin{equation}
\label{3.2}
\xi\subset N(\epsilon,O(q^2,f))\quad\textrm{and}\quad O(q^2,f)\subset N(\epsilon,\xi) 
\tag{3.2}
\end{equation}
and the inequality $\textrm{dist}(q_1^1,q_1^2)<\epsilon$ the following holds:
\begin{equation}
\label{3.3}
q^2\in W^s(p),
\tag{3.3}
\end{equation}
\begin{equation}
\label{3.4}
\textrm{dist}(q^1_k,q^2_k)\leq R\quad\textrm{for}\ k\geq1.
\tag{3.4}
\end{equation} 
\end{lemma}

In fact, the lemma states the following: Let $p$ be a hyperbolic periodic point. Then, if a pseudotrajectory $\xi=\{x_k\}$ of a certain type is ''close'' to the trajectory $O(p,f)$ for all sufficiently large positive $k$ and is ''far'' from the trajectory $O(p,f)$ for all negative $k$ with sufficiently large absolute values, then any exact trajectory that orbitally shadows the pseudotrajectory $\xi$ is contained in the stable manifold of the point $p$.

\begin{proclaim}{Corollary}
Under the conditions of Lemma 1, suppose that $q^1\in W^u(p)$. There exists a number $0<\epsilon<\epsilon_0$ such that if a sequence $\xi=\{x_k\}$ satisfies the relations 
\begin{equation}
\label{3.5}
x_k=q^1_k=f^k(q^1)\qquad\textrm{for}\ k\leq1,\quad\ x_k\notin N(\epsilon_0,O(p,f))\qquad\textrm{for}\ k>1,
\tag{3.5}
\end{equation}
then for any point $q^2$ which satisfies the inequality $\textrm{dist}(q_1^1,q_1^2)<\epsilon$ and inclusions~$(\ref{3.2})$ the following holds:
\begin{equation}
\label{3.6}
q^2\in W^u(p),
\tag{3.6}
\end{equation}
\begin{equation}
\label{3.7}
\textrm{dist}(q^1_k,q^2_k)\leq R\quad\textrm{for}\ k\leq1.
\tag{3.7}
\end{equation}  
\end{proclaim}

\begin{proof}[Proof of Lemma 1] At first, we need to choose a sufficiently small $\epsilon$. We do it in several steps. 

Choose a number $\epsilon_1<\epsilon_0$ such that
\begin{itemize}
\item if $O(p,f)=\{p_0,\ldots,p_{m_p-1}\}$, then the neighborhoods $$N(\epsilon_1,p_0),\ N(\epsilon_1,p_1),\ldots, N(\epsilon_1, p_{m_p-1})$$ 
are disjoint; 
\item if a positive semitrajectory $O_{+}(x,f)$ is contained in the set $N(\epsilon_1,O(p,f))$, then $x\in W^s(p)$;
\item $\left(N(\epsilon_1,O(p,f))\cup f(N(\epsilon_1,O(p,f)))\right)\cap N(\epsilon_1,x_k)=\emptyset$ for $k\leq1$;
\item there are no points of the sequence $\xi$ on the boundary of the set $N(\epsilon_1,O(p,f))$.
\end{itemize}
Let $n>1$ be the minimal number such that the points $x_k$ are contained in the neighborhood $N(\epsilon_1,O(p,f))$ for $k\geq n$. Choose a number $\epsilon_2<\epsilon_1$ such that
\begin{itemize}
\item the following neighborhoods are disjoint:
$$N(\epsilon_2,O(p,f))\quad\textrm{and}\quad N(\epsilon_2,x_k)\ \textrm{for}\ 1\leq k\leq n;$$
\item if there exist numbers $1\leq k_1\leq n$ and $1\leq k_2$ such that
$$x_{k_1}=q^1_{k_1}\notin N(\epsilon_1,O(p,f))\quad\textrm{and}\quad x_{k_2}=q^1_{k_2}\in N(\epsilon_1,O(p,f)),$$
then $N(\epsilon_1,O(p,f))\cap N(\epsilon_2,x_{k_1})=\emptyset$ and $N(\epsilon_2,x_{k_2})\subset N(\epsilon_1,O(p,f))$. 
\end{itemize}
Let $m\geq n$ be the minimal number such that the point $x_m$ is contained in the set $N(\epsilon_2/3,O(p,f))$. Choose a number $\epsilon_3<\epsilon_2/3$ such that the neighborhoods $N(\epsilon_3,x_k)$ are disjoint for $1\leq k\leq m$. Choose a number $\epsilon<\epsilon_3$ such that
$$f^{j-k}(N(\epsilon,x_k))\subset N(\epsilon_3,x_j)\qquad\textrm{for}\ j,k\in\{1,\ldots,m\}.$$
Let us show that the number $\epsilon$ chosen above has the desired properties. Suppose that relation $(\ref{3.2})$ holds for a point $q^2$ and for the sequence $\xi$.

Since $q^2_1\in N(\epsilon,x_1)$, the choice of $\epsilon$ implies that 
\begin{equation}
\label{3.8}
q^2_k\in N(\epsilon_3,x_k)\qquad\textrm{for}\ 1\leq k\leq m.
\tag{3.8}
\end{equation}
Thus, by inclusion $(\ref{3.8})$ and our notations, the point $q_2^m$ is contained in the set $N(\epsilon_2,O(p,f))$. In order to obtain inclusion $(\ref{3.3})$, it is sufficient to prove the following inclusion
\begin{equation}
\label{3.9}
q^2_k\in N(\epsilon_1,O(p,f))\qquad\textrm{for}\ k\geq m.
\tag{3.9}
\end{equation}

Suppose that inclusion $(\ref{3.9})$ does not hold, i.e., there exists the minimal number $r>m$ such that 
\begin{equation}
\label{3.10}
q^2_r\notin N(\epsilon_1,O(p,f)).
\tag{3.10}
\end{equation}
From relation $(\ref{3.2})$ and the choice of $\epsilon_2$ and $\epsilon_3$ it follows that there exists a number $k<n$ such that the point $q^2_r$ is contained in the set $N(\epsilon,x_k)$. Two cases are possible:

Case 1: $k\leq 1$. By the choice of $r$, $q^2_{r-1}\in N(\epsilon_1,O(p,f))$. But then, by relation~$(\ref{3.10})$,
$$f(N(\epsilon_1,O(p,f)))\cap N(\epsilon_1,x_k)\neq\emptyset\qquad\textrm{for}\ k\leq1,$$
what contradicts to the choice of $\epsilon$.

Case 2: $1<k< n$. There exists the maximal number $1\leq k^\prime<k$ such that the point $x_{k^{\prime}}$ does not belong to the set $N(\epsilon_1,O(p,f))$. But then, by the choice of $\epsilon$, 
$$q^2_{r-t}\in N(\epsilon_3,x_{k-t})\quad\textrm{for }0\leq t\leq k-k^{\prime},$$
and, hence, by the choice of $\epsilon_2$,
$$q^2_{r-(k-k^\prime)}\notin N(\epsilon_1, O(p,f))\quad\textrm{and}\quad q^2_{r-t}\notin N(\epsilon_2,O(p,f))\ \ \ \textrm{for}\ 0\leq t\leq k-k^\prime.$$
Thus, $r-(k-k^\prime)>m$ (since, otherwise, one could find a number $1\nobreak\leq\nobreak t\nobreak\leq\nobreak k\nobreak -\nobreak k^{\prime}$ such that the point $q^2_m=q^2_{r-t}$ belongs to the set $N(\epsilon_2,O(p,f))$). Hence, the number $r$ is not minimal, and we get the desired contradiction.

We proved inclusion $(\ref{3.9})$ and, hence, inclusion $(\ref{3.3})$ too. By relations $(\ref{3.8})$ and $(\ref{3.9})$, and since $q^1_k\in N(\epsilon_1,O(p,f))$ for $k\geq m$, inequality $(\ref{3.4})$ holds for the points $q^1$ and $q^2$. Lemma 1 is proved.
\end{proof}          

\section{Reduction of the proof of Theorem $\textrm{A}^{\prime}$ to two cases, the proof in Case~(A1)}

Two following cases are possible:
\begin{enumerate}
\item[(A1)] There exist two hyperbolic periodic, lying on different fibres, points $r_1$ and $r_2$ with one-dimensional unstable and stable manifolds, respectively, such that this manifolds intersect.
\item[(A2)] For any hyperbolic periodic, lying on different fibres, points $r_1$ and $r_2$ with one-dimensional unstable and stable manifolds, respectively, this manifolds do not intersect.
\end{enumerate}

Now we can give a more detailed outline of the proof scheme in the two cases under consideration.

In Case (A1), we construct the pseudotrajectory in the following way: It  
includes a part of an exact heteroclinic trajectory from $r_1$ to $r_2$, then it ''leaps'' (sufficiently close to the point $r_2$) to a trajectory that lies on the fibre of the point $r_2$ and ''goes away'' from the point $r_2$ by its unstable manifold. Let us emphasize that Main Lemma can be applied not only to the pseudotrajectory, but also to its projection to the base (the Smale horseshoe). To get a contradiction, suppose that the constructed pseudotrajectory is orbitally shadowed by an exact trajectory. Next, we apply Main Lemma to the projection of the trajectory to the base and see that this trajectory lies on the fibre of some trajectory ''going to'' the point $r_2$. On the other hand, since the pseudotrajectory ''goes away'' from the point $r_1$, the exact trajectory should be contained in the unstable manifold of the point $r_1$. In addition, since the Bernoulli shift is expansive, the projection to the base of the heteroclinic trajectory used in our construction should coincide with the projection of the exact shadowing trajectory. Finally, since the projection of a local unstable manifold is one-to-one in a small neighborhood of the point $r_1$, and we know the point in the base, the exact trajectory should precisely coincide with the heteroclinic trajectory used in our construction. We see the contradiction: the ''final'' part of the pseudotrajectory (near the point $r_2$) is not shadowed. The phase portrait of a mild skew product in Case (A1) is depicted in Fig. 1.

\bigskip \bigskip
\bigskip 
\bigskip
\bigskip \bigskip
{\small{
\begin{picture}(100,120)
\put(330,100){\circle*{3}}
\put(330,100){\vector(-3,-1){160}}
\put(335,94){$r_1$}
\put(50,100){\circle*{3}}
\put(50,100){\line(0,1){70}}
\put(50,100){\line(0,-1){70}}
\put(213,46){\vector(-3,1){90}}
\put(213,46){\line(-3,1){160}}
\put(191,54){\circle*{3}}
\put(198,42){$x$}
\put(37,92){$r_2$}
\put(0,164){$prr_2\times S^1$}
\put(50,100){\vector(0,1){30}}
\put(50,100){\vector(0,-1){30}}
\put(50,115){\circle*{3}}
\put(39,111){$y$}
\put(260,60){$W^u(r_1)$}
\put(120,60){$W^s(r_2)$}
\put(330,10){Fig. 1}
\end{picture}
}
}
\\

In Case (A2), using the technique of mild skew products, we construct a pseudotrajectory that ''goes from'' and ''goes to'' hyperbolic periodic points with one-dimensional unstable and stable manifolds, respectively, in the initial and final phases, respectively, and does not approach this points in the intermediate phase. Then, by Main Lemma, if an exact trajectory orbitally shadows the pseudotrajectory, then it should be a heteroclinic trajectory connecting this points, and we get the desired contradiction.   

Case (A1) is proved in this section, and Case (A2) is considered in the remaining sections.

Suppose that a mild skew product $G$ satisfies the conditions of Case (A1). Choose points $x\in W^u(r_1)\cap W^s(r_2)$ and $y\in (pr r_2\times S^1)\cap W^u(r_2)$ (cf.,~Fig.~1). At first, we need to choose a sufficiently small number $\varepsilon$. We do it in several steps.

Choose a number $\varepsilon_0>0$ such that
\begin{enumerate}
\item[1)] the statement of Main Lemma corollary holds for the mapping $G$, some point of the trajectory $O(x,G)$ and the point $r_1$;
\item[2)] the statement of Main Lemma holds for the Bernoulli shift $\sigma$, some point of the trajectory $O(pr x,\sigma)$ and the point $pr r_2$ (in this two items the number $R$ is assumed to be a sufficiently small, but preliminary fixed number);
\item[3)] the point $y$ does not belong to the set $N(2\varepsilon_0,O(x,G))$;
\item[4)] the restriction of the projection $pr$ onto a local unstable manifold $W_{\varepsilon_0}^u(r_1)$ is a one-to-one mapping.
\end{enumerate}   

Let us explain items 1) and 2). Strictly speaking, both Main Lemma and its corollary can not be applied directly in our case, since the spaces $\Sigma^2$ and $\Sigma^2\nobreak\times\nobreak S^1$ are not manifolds, and the above-mentioned statements were proved only for the manifolds. However, in item 1), we can consider the diffeomorphism $F$ of the manifold $M$ that corresponds to the mild skew product $G$ (in the sense of Remark 1), and points $\bar{x}$ and $\bar{r}_1$ which are the analogs of the points $x$ and $r_1$ for the diffeomorphism $F$. There exist a number $\epsilon_0$ and a point $x_0\in O(x,G)$ that satisfy the analog of condition $(\ref{3.5})$. Choose a point $q^2\in\Sigma^2\times S^1$ that satisfies the analog of relation~$(\ref{3.2})$ for some small $\epsilon$ and the inequality $\textrm{dist}(G(q^2),G(x_0))<\epsilon$. By Theorem 2, there exists a homeomorphism $h_F$ conjugating the restriction $F|_{\Delta}$ (where the set $\Delta$ was defined in the conditions of Theorem 2) and the mild skew product $G$. Since the homeomorphisms $h_F$ and $h_F^{-1}$ are homeomorphisms of compact metric spaces, they are uniformly continuous. Hence, the analog of condition~$(\ref{3.5})$ holds for some number $\bar{\epsilon}_0$ and point $\bar{x}_0=h_F^{-1}(x_0)$. Moreover, all conditions of Corollary are satisfied for the point $\bar{q}^2=h_F^{-1}(q^2)$, an analog of the point $q^2$, the point $\bar{x}_0$, sufficiently small number $\bar{\epsilon}$, and the mapping $F$. Thus, Corollary can be applied to the diffeomorphism $F$, i.e., the analogs of relations $(\ref{3.6})$ and $(\ref{3.7})$ hold for sufficiently small numbers $\bar{R}$ and $\bar{\epsilon}$, and the points $\bar{x}_0$ and $\bar{q}^2$. Hence, by the uniform continuity of the homeomorphism $h_F$, the analogs of relations $(\ref{3.6})$ and $(\ref{3.7})$ hold for sufficiently small numbers $R$ and $\epsilon$, and the points $x_0$~and~$q^2$. 

Hence, the statement of Corollary holds for the mapping $G$. In item 2), we can consider the diffeomorphism $g:S^2\mapsto S^2$ (which was fixed above) having the Smale horseshoe and apply similar reasoning. Thus, the statement of Lemma 1 can be applied to the Bernoulli shift $\sigma$ too. 

Let us explain item 4). Consider the diffeomorphism $F:M\nobreak\mapsto\nobreak M$ corresponding to the mild skew product $G$. By Theorem 2, the diffeomorphism $F$ has the local maximal invariant set $\Delta$ homeomorphic to $\Sigma^2\times S^1$ such that $F|_{\Delta}$~and~$G$ are topologically conjugate. The sets homeomorphic to $S^1$ and corresponding to the fibres of the mild skew product $G$ are called the center fibres. By Remark 1, any periodic point $r_1$ of the mild skew product $G$ is assigned to some hyperbolic periodic point $\bar{r}_1$ of the diffeomorphism $F$. Since $r_1$ is a point of type $(2,1)$, a local unstable manifold of the point $\bar{r}_1$ with respect to the diffeomorphism $F$, i.e, a set $W^u_{\textrm{loc}, F}(\bar{r}_1)$, is just a finite union of ''intervals''. The angle between the unstable space of the diffeomorphism $F$ at the point~$\bar{r}_1$ and the corresponding central fibre is not equal to zero. It is shown in the papers \cite{3,5,6} that the central fibres of the diffeomorphism $F$ are $C^1$-close to the corresponding fibres of the diffeomorphism $g\times\textrm{id}_{S^1}$, i.e., to the circles. Hence, the intersection of the set $W^u_{\textrm{loc}, F}(\bar{r}_1)$ and any central fibre consists of no more than one point. Consequently, the intersection of the set $W^u_{\textrm{loc}, G}(r_1)$ and any fibre also consists of no more than one point. Thus, the projection to the base is a one-to-one mapping on the set $W^u_{\textrm{loc}, G}(r_1)$.

We need the following two lemmas: 

\begin{lemma}
There exists a number $R$ such that if $\sigma:\Sigma^2\mapsto\Sigma^2$ is the Bernoulli shift, and points $q^1,q^2\in\Sigma^2$ satisfy the inequality
$$d_{\Sigma^2}(q^1_k,q^2_k)\leq R\quad\textrm{for }k\in\textbf{Z},$$
then $q^1=q^2$ (recall that $q^j_k=\sigma^k(q^j)$).  
\end{lemma}      

In fact, Lemma 2 means that the Bernoulli shift is expansive. The proof of this fact is given, e.g., in the book~\cite{2}.

\begin{lemma}
Under the conditions of Case (A1), let $R$ be the number given by Lemma 2. There exists an $\varepsilon<\varepsilon_0$ such that if the relations 
$$d_{\Sigma^2}(q_1^1,q_1^2)<\epsilon,$$ 
$$O(q^1,\sigma)\subset N(\varepsilon,O(q^2,\sigma))\quad\textrm{and}\quad O(q^2,\sigma)\subset N(\varepsilon, O(q^1,\sigma))$$
hold for two points $q^1=pr x$ and $q^2$ from $\Sigma^2$ such that $q^1,q^2\in W^u(pr r_1)\cap\allowbreak\cap W^s(pr r_2)$, then $\textrm{dist}(q^1_k,q^2_k)\leq R$ for $k\in\textbf{Z}$.
\end{lemma} 

In other words, Lemma 3 means that, under the conditions of Case (A1), if two heteroclinic trajectories which ''go from'' $pr r_1$ ''to'' $pr r_2$ are orbitally close, then they are pointwise close.

\begin{proof}
By the choice of $\varepsilon_0$, the statement of Lemma 1 corollary holds for the point $pr r_1$, the mapping $\sigma$ and some point $q^1_{k_1-1}=\sigma^{k_1-1}(q^1)\in O(pr x,\sigma)$, and the statement of Lemma 1 holds for the point $pr r_2$, the mapping $\sigma$ and some point $q^2_{k_2-1}\in O(pr x,\sigma)$ (with $R$ given by Lemma 2). Hence, if $d_{\Sigma^2}(q^1_{k_1},q^2_{k_1})<\allowbreak <\varepsilon_0$, then, by inequalities $(\ref{3.4})$,
$$d_{\Sigma^2}(q^1_{k},q^2_{k})\leq R\qquad\textrm{for}\ k\leq k_1.$$
Similarly, if $d_{\Sigma^2}(q^1_{k_2}, q^2_{k_2})<\varepsilon_0$, then 
$$d_{\Sigma^2}(q^1_{k},q^2_{k})\leq R\qquad\textrm{for}\ k\geq k_2.$$

Choose a number $\varepsilon<\varepsilon_0$ such that the inequalities $d_{\Sigma^2}(q_1^1,q_1^2)<\varepsilon$ imply the inequalities $$d_{\Sigma^2}(q^1_k,q^2_k)<\varepsilon_0\quad\textrm{for}\ k\ \textrm{between 1 and}\ k_1,$$
$$d_{\Sigma^2}(q^1_k,q^2_k)<\varepsilon_0\quad\textrm{for}\ k\ \textrm{between 1 and}\ k_2.$$
The number $\varepsilon$ has the desired properties, i.e., the statement of Lemma 3 holds for this number. 
\end{proof}

Let $\varepsilon$ be the number whose existence was proved in Lemma 3. Choose an arbitrary number $d<\varepsilon$. Now we construct the pseudotrajectory discussed at the beginning of Sec. 4. Choose numbers $k_1$ and $k_2$ such that $$x_{k_1+1}\in N(d/2,O(r_2,G))\quad\textrm{and}\quad y_{k_2}\in N(d/2,O(r_2,G)).$$
We construct the $d$-pseudotrajectory $\xi=\{\xi_k\}$ in the following way (cf. Fig.~1):
$$\xi_k=x_k\textrm{ for }k\leq k_1,\quad\xi_k=y_{k-k_1+k_2-1}\textrm{ for }k>k_1.$$ 
Suppose that the mild skew product $G$ has OSP. Then there exists the point $q$ such that relation~$(\ref{1.1})$ holds for it and the pseudotrajectory $\xi$.

By the choice of $\varepsilon$, the corollary of Lemma 1 holds for the point $r_1$ and the pseudotrajectory $\xi$. Hence, 
\begin{equation}
\label{4.1}
q\in W^u(r_1).
\tag{4.1}
\end{equation}

Consider the sequence $pr \xi$.  The points of the sequence $pr \xi$ coincide with the corresponding points of the trajectory $O(pr r_2,\sigma)$ before the intersection with the $d/2$-small neighborhood of the point $pr r_2$. Therefore, by relation $(\ref{1.1})$,
\begin{equation}
\label{4.2}
O(pr x,\sigma)\subset N(\varepsilon, O(pr q,\sigma))\quad\textrm{and}\quad O(pr q,\sigma)\subset N(\varepsilon, O(pr x,\sigma))
\tag{4.2}
\end{equation}
for a sufficiently small $d$.   
Thus, relation~$(\ref{4.2})$ and the analog of relation $(\ref{3.1})$ hold for the trajectory of the point $pr x$. Hence, by the choice of $\varepsilon$ and by relation~$(\ref{4.1})$,
$$pr q\in W^u(pr r_1)\cap W^s(pr r_2).$$
 
All conditions of Lemma 3 hold for the points $pr x$ and $pr q$ (of course, we can assume that $\textrm{dist}(x_1,q_1)<\varepsilon$). Hence, $$d_{\Sigma^2}(\sigma^k(pr x),\sigma^k(pr q))\leq R\qquad\textrm{for}\ k\in\textbf{Z}.$$
By Lemma 2, $$pr q_k=pr x_k\quad\textrm{for }k\in\textbf{Z}.$$
Since $x, q\in W^u(r_1)$, one can find a number $K$ such that the points $x_K$ and $q_K$ belong to a local unstable manifold of the point $r_1$ that can be projected to the base injectively (cf., the choice of $\varepsilon$, item 4)). Hence, the equality $pr q_K=pr x_K$ implies the equality $q_K=x_K$, and the latter one implies the equality $q=x$. Consequently, by relation $(\ref{1.1})$, the following inclusions hold:
$$O(x,G)\subset N(\varepsilon,\xi)\quad\textrm{and}\quad \xi\subset N(\varepsilon,O(x,G)),$$
what contradicts to the construction of the sequence $\xi$. The derived contradiction means that our assumptions are wrong, and $G\notin\textrm{OSP}$ in Case (A1). 

\section{Start of the proof in Case (A2): auxiliary lemmas}

Two auxiliary lemmas on properties of skew products necessary for the proof in Case (A2) will be formulated and proved in this section. Let us introduce corresponding notations. 

Consider the step skew product $G_0$ generated by the diffeomorphisms $g_0$ and $g_1$. By Theorem 1, the step skew product $G_0$ has an infinite number of hyperbolic periodic points of type (1,2) lying on different fibres and an infinite number of hyperbolic periodic points of type (2,1) lying on different fibres. Note that any infinite set of lying on different fibres periodic points in the space $\Sigma^2\times S^1$ contains points of arbitrary large periods. Choose four hyperbolic periodic, lying on different fibres, points of the step skew product $G_0$: points $p_1$ and $p_3$ of type (2,1), and points $p_2$ and $p_4$ of type (1,2).

Recall that any finite sequence of zeros or ones is called a word. The sequence $pr p_j$, where $j\in\{1,\ldots,4\}$, is periodic, i.e, some word $\omega_j$ of length $T_j$ is periodically repeated in it. We can assume that the word $\omega_j$ is the word of minimal length, i.e., the number $T_j$ is the main period of the point $pr p_j$ with respect to the Bernoulli shift $\sigma$. Without loss of generality, we assume that $T_j>2\ \textrm{for}\ j\in\{1,\ldots,4\}$, i.e., the word $\omega_j$ contains both $0$ and $1$. In addition, without loss of generality, we assume that
\begin{equation}
\label{5.1}
T_3\leq\min(T_1,T_2)\qquad\textrm{and}\qquad T_4\leq\min(T_1,T_2).
\tag{5.1}
\end{equation}
By definition, put
\begin{equation}
\label{5.2}
T=\max(T_1,T_2).
\tag{5.2}
\end{equation}

We can assume that the neighborhoods $W_0$ and $W_1$ from Theorem $\textrm{A}^{\prime}$ are so small that the points $p_j$ are preserved for any mild skew product (which is, in fact, a perturbation of the step skew product $G_0$) from Theorem $\textrm{A}^{\prime}$. It means that the analogs of this points have the same periods and types. In particular, the number $T$ does not depend on the choice of a mild skew product $G$. Let~$\delta$ be the maximal radius of the neighborhoods $W_0$ and $W_1$. We can assume that it is an arbitrary small (and dependent on $T$) but fixed number. Main restrictions on the size of $\delta$ will be imposed further, in Sec. 7.

We denote by the same symbols $p_j$ the hyperbolic periodic points of the mild skew product $G$ corresponding to the points $p_j$ of the step skew product $G_0$. As it was noted above, the periods $T_j$ and the types of the points $p_j$ have not changed. Suppose that, as before, $\omega_j$ ($j\in\{1,\ldots,4\}$) are periodically repeating words of the sequences $pr p_j$.

Let us define the cylinder neighborhoods $U_j$ of the points $p_j$ ($j\in\{1,2\}$) by the formula
$$U_j:=\left\{\omega=\ldots\omega_{j}\omega_{j}|\omega_{j}\omega_{j}\ldots\right\}\times S^1.$$
In the previous formula, the dots denote arbitrary symbols, and the meaning of the sign $|$ was explained above (cf., Sec. 2). The word $\omega_j$ is repeated four times: two times before the zero position and two times after it. Define the cylinder neighborhoods of the trajectories $O(p_j,G)$ ($j\in\{1,2\}$) by the formula 
$$V_j:=U^0_j\cup U^1_j\cup\ldots\cup U^{T_j-1}_j,$$
where the set $U_j^k$ ($0\leq k\leq T_j-1$) is defined similarly with the set $U_j$, only the word $\omega_j$ is changed by the word $\sigma^{k}(\omega_j)$, i.e., the corresponding cyclic permutation of the word $\omega_j$.       

\begin{lemma}
Under our conditions,
\begin{equation}
\label{5.3}
O(p_j,G)\cap V_t=\emptyset\qquad\textrm{for}\ j\in\{3,4\},\ t\in\{1,2\},
\tag{5.3}
\end{equation}
i.e., the trajectories of the points $p_3$ and $p_4$ do not intersect the cylinder neighborhoods $V_1$ and $V_2$ of the trajectories $O(p_1,G)$ and $O(p_2,G)$.
\end{lemma}

\begin{proof}
Without loss of generality, we prove relation $(\ref{5.3})$ for the point $p_3$ and the set $U_1$. To get a contradiction, assume that this relation does not hold. It means that there exists a number $K$ such that the word $\omega_1\omega_1\omega_1\omega_1$ takes the positions from $K-2T_1$ to $K+2T_1-1$ in the sequence $pr p_3$. In addition, by relation $(\ref{5.1})$, the word $\omega_1$ is longer than the word $\omega_3$. Consider the word $\omega_1$ starting from the $K$-position of the sequence $pr p_3$. The first $T_3$ symbols of this word are a cyclic permutation of the word $\omega_3$. Denote this permutation by $\bar{\omega}_3$. Hence, the word $\omega_1$ is covered by $m$-times repeated words $\bar{\omega}_3$ plus an ''addition'': $r$ first symbols from $\bar{\omega}_3$ ($0\leq r<T_3$).

However, the second word $\omega_1$ (the one that starts from the $(K+T_1)$-position of the sequence $pr p_3$, cf. Fig. 2) is also covered by words $\bar{\omega}_3$. On the one hand, it should start from $\bar{\omega}_3$ (since the first and the second words $\omega_1$ coincide); on the other hand, it should start from last $T_3-r$ symbols of $\bar{\omega}_3$ (cf. Fig. 2). It means that if we swap first $r$ symbols and last $T_3-r$ symbols in the word $\bar{\omega}_3$, then the word $\bar{\omega}_3$ will not change, i.e., $\sigma^{r}(pr p_3)=pr p_3$. Hence, $r=0$, but then the word $\omega_1$ is the $m$-times repeated word $\bar{\omega}_3$. Thus, the trajectories of the points $p_1$ and $p_3$ intersect, and we get a contradiction with the choice of the points $p_1$ and $p_3$. Relation~$(\ref{5.3})$ is proved. 
\end{proof}

{\small{
\begin{picture}(100,120)
\put(0,50){\line(1,0){350}}
\put(100,50){\line(0,1){40}}
\put(100,50){\line(0,-1){40}}
\put(45,35){$\omega_1$}
\put(230,50){\line(0,1){40}}
\put(230,50){\line(0,-1){40}}
\put(100,30){\vector(1,0){130}}
\put(230,30){\vector(-1,0){130}}
\put(170,15){$\omega_1$}
\put(140,50){\line(0,1){22}}
\put(100,65){\vector(1,0){40}}
\put(140,65){\vector(-1,0){40}}
\put(115,70){$\bar{\omega}_3$}
\put(180,50){\line(0,1){22}}
\put(140,65){\vector(1,0){40}}
\put(180,65){\vector(-1,0){40}}
\put(155,70){$\bar{\omega}_3$}
\put(220,50){\line(0,1){22}}
\put(180,65){\vector(1,0){40}}
\put(220,65){\vector(-1,0){40}}
\put(195,70){$\bar{\omega}_3$}
\put(260,50){\line(0,1){22}}
\put(220,65){\vector(1,0){40}}
\put(260,65){\vector(-1,0){40}}
\put(300,50){\line(0,1){22}}
\put(260,65){\vector(1,0){40}}
\put(300,65){\vector(-1,0){40}}
\put(270,50){\line(0,-1){15}}
\put(230,40){\vector(1,0){40}}
\put(270,40){\vector(-1,0){40}}
\put(248,30){$\bar{\omega}_3$}
\put(230,20){\vector(1,0){120}}
\put(350,50){\line(0,-1){40}}
\put(350,50){\line(0,1){40}}
\put(350,20){\vector(-1,0){120}}
\put(292,11){$\omega_1$}
\put(330,-10){Fig. 2}
\put(86,36){$K$}
\end{picture}
}
}
\\

\begin{lemma}
Suppose that $m>4T$ and a sequence $\beta\in\Sigma^2$ is such that a word $$\omega=\alpha_1\omega_3\ldots\omega_3\alpha_2$$
is repeated in it periodically, where the word $\omega_3$ is repeated precisely $m$ times in the formula, and the words $\alpha_1$ and $\alpha_2$ (whose length is more than $4T$) can not contain less than $T$ zeros in a row (however, it is allowed not to contain any zeros at all); then
\begin{equation}
\label{5.4}
O(\beta,\sigma)\cap pr V_t=\emptyset\qquad\textrm{for}\ t\in\{1,2\}.
\tag{5.4}
\end{equation}
In other words, the trajectory of the sequence $\beta$ with respect to the Bernoulli shift $\sigma$ that includes the word $\omega$ periodically does not intersect the sets $V_1$~and~$V_2$.  
\end{lemma}

\begin{proof}
Without loss of generality, we prove that relation $(\ref{5.4})$ holds for the point $\beta$ and the set $pr U_1$. To get a contradiction, suppose the contrary. It means that there exists a number $K$ such that the word $\omega_1\omega_1\omega_1\omega_1$ takes the positions from $K-2T_1$ to $K+2T_1-1$ in the sequence $\beta$. Two cases are possible:

a) The $K$-position is ''included'' in the word $\alpha_1$ (in the case when it is included in the word $\alpha_2$, we can apply the same reasoning). By construction, the word $\omega_1$ contains both zeros and ones. Hence, the word $\omega_1$ can not contain more than $T_1-1$ zeros. This fact and the fact that the word $\omega_1\omega_1$ is contained in the word $\alpha_1$ contradict to the properties of $\alpha_1$. 

b) The $K$-position is ''included'' in the word $\omega_3\ldots\omega_3$. Since the word $\omega_3$ is repeated $m$ times and $m>4T$, the word $\omega_1\omega_1$ is covered by a finite number of words $\omega_3$. Further reasoning for deriving the contradiction is similar with the proof of Lemma 4.

We got the contradiction in both possible cases. Hence, our assumptions are wrong. Lemma 5 is proved. 
\end{proof}

\section{Reduction of the proof in Case (A2) to Lemma 6}

Recall that a sketch  of the proof in Case (A2) was outlined at the beginning of Sec. 4. In this section we give the proof in Case (A2) with an exception of one lemma. 

Choose a mild skew product $G$ from Theorem~$\textrm{A}^{\prime}$ and suppose that it satisfies the conditions of Case (A2) (cf. the definition at the begining of Sec. 4). The following lemma plays a key role in the proof of Case (A2). In fact, it states that it is possible to construct ''as precise as we want'' pseudotrajectories with the required properties (they should ''go from'' the point $p_1$ and ''go to'' the point $p_2$, and their ''intermediate part'' should be ''separated'' from the trajectories $O(p_1,G)$ and $O(p_3,G)$).

{\small{
\begin{picture}(100,180)
\put(15,70){\circle*{3}}
\put(7,60){$p_2$}
\put(15,70){\vector(0,1){90}}
\put(15,70){\vector(0,-1){90}}
\put(5,70){\line(0,1){90}}
\put(5,70){\line(0,-1){90}}
\put(25,70){\line(0,1){90}}
\put(25,70){\line(0,-1){90}}
\put(30,-20){$V_2$}
\put(340,70){\circle*{3}}
\put(332,60){$p_1$}
\put(330,160){\line(0,-1){90}}
\put(340,160){\vector(0,-1){60}}
\put(330,-20){\line(0,1){90}}
\put(340,-20){\vector(0,1){60}}
\put(340,70){\line(0,1){90}}
\put(340,70){\line(0,-1){90}}
\put(350,70){\line(0,1){90}}
\put(350,70){\line(0,-1){90}}
\put(320,-20){$V_1$}
\put(85,70){\circle*{3}}
\put(85,70){\vector(0,1){90}}
\put(85,70){\vector(0,-1){90}}
\put(85,70){\line(-1,0){70}}
\put(85,70){\vector(-1,0){50}}
\put(55,70){\circle*{3}}
\put(60,60){$z$}
\put(75,60){$p_4$}
\put(250,70){\circle*{3}}
\put(250,70){\line(1,0){90}}
\put(250,70){\line(0,1){90}}
\put(250,70){\line(0,-1){90}}
\put(250,70){\circle{40}}
\put(250,160){\vector(0,-1){50}}
\put(250,-20){\vector(0,1){50}}
\put(255,60){$p_3$}
\put(340,70){\vector(-1,0){60}}
\put(300,70){\circle*{3}}
\put(300,60){$x$}
\put(235,70){\circle*{3}}
\put(235,70){\vector(0,1){90}}
\put(235,70){\vector(0,-1){90}}
\put(235,70){\vector(-1,0){120}}
\put(235,70){\line(-1,0){210}}
\put(239,63){$s$}
\put(180,70){\circle*{3}}
\put(180,60){$y$}
\put(85,70){\line(4,1){70}}
\put(155,87){\circle*{3}}
\put(155,87){\vector(-4,-1){30}}
\put(155,87){\vector(0,1){70}}
\put(155,87){\vector(0,-1){60}}
\put(158,80){$s$}
\put(155,87){\line(0,-1){100}}
\put(270,85){$N(d,p_3)$}
\put(330,-40){Fig. 3}
\end{picture}
}
}
\\

\bigskip
\bigskip

\begin{lemma}
Under our conditions, if $\delta$ is sufficiently small, then
\begin{enumerate}
\item[(6.a)] the one-dimensional unstable manifold of the point $p_1$ and the two-dimensional stable manifold of the point $p_3$ intersect;
\item[(6.b)] the two-dimensional unstable manifold of the point $p_4$ and the one-dimensional stable manifold of the point $p_2$ intersect;
\item[(6.c)] given any $d$, there exists a hyperbolic periodic point $s$ such that
\begin{itemize}
\item $s\in N(d,p_3)$, and the unstable manifold of the point $s$ is one-dimensional,
\item the trajectory $O(s,G)$ does not intersect the sets $V_1$ and $V_2$; 
\end{itemize}
\item[(6.d)] there exists a point $y\in W^u(s)\cap W^s(p_4)$ whose trajectory $O(s,G)$ does not intersect the sets $V_1$ and $V_2$. 
\end{enumerate}
\end{lemma}

Lemma 6 will be proved in Sec. 7. Choose points $x\in W^u(p_1)\cap W^s(p_3)$ and $z\in W^u(p_4)\cap W^s(p_2)$.
The phase portrait of the mild skew product $G$ is depicted in Fig. 3. For convenience, the symbol $s$ denotes all points of the trajectory $O(s,G)$ from Lemma 6. Let us show how to finish the proof for Case~(A2) using Lemma 6. At first, we need to choose a sufficiently small number $\varepsilon$. We do it in several steps.

The points $x$ and $z$ can be chosen so close to the points $p_1$ and $p_2$ that there exists a number $\varepsilon_0$ such that 
$$N(\varepsilon_0,O_{-}(x,G))\subset V_1,\quad N(\varepsilon_0,O_{+}(z,G))\subset V_2;$$
$$N(\varepsilon_0,O_{-}(x,G))\cap N\left(\varepsilon_0, O(z,g)\cup O(y,G)\right)=\emptyset;$$
 $$N(\varepsilon_0,O_{+}(z,G))\cap N\left(\varepsilon_0,O(x,G)\cup O(y,G)\right)=\emptyset,$$
where $y$ is an arbitrary point from Lemma 6. There exists a number $\varepsilon<\varepsilon_0/3$ such that 
\begin{itemize}
\item the statement of Main Lemma holds for the mapping $G$, the point $z$ and the point $p_2$;
\item the statement of Main Lemma corollary holds for the mapping $G$, the point $x$ and the point $p_1$. 
\end{itemize}   
It was shown in Sec. 4 (cf., the choice of $\varepsilon_0$), that both Main Lemma and its corollary can be applied to the mild skew products. 

Now, when $\varepsilon$ is chosen, we can construct the desired pseudotrajectory. Choose an arbitrary number $d<\varepsilon$. Suppose that $s$ is a point corresponding to the number $d/3$, and $y$ is a point from item (6.d) corresponding to the point $s$ whose existence is proved in Lemma 6. There exist numbers $k_1$, $k_2$, $k_3$ and $k_4$ such that 
$$x_{k_1+1}\in N(d/3,p_3),\quad y_{k_2}\in N(d/3,s),$$
$$y_{k_3+1}\in N(d/2,p_4),\quad z_{k_4}\in N(d/2,p_4).$$
Let us construct the $d$-pseudotrajectory $\xi=\{\xi_k\}$ in the following way:
$$\xi_k=x_k\textrm{ for }k\leq k_1,\quad\xi_{k}=y_{k-k_1-1+k_2}\textrm{ for }k_1<k\leq k_1+1+k_3-k_2,$$
$$\xi_k=z_{k-k_1-2-k_3+k_2+k_4}\textrm{ for }k>k_1+1+k_3-k_2.$$

Suppose that the mild skew product $G$ has OSP, i.e., there exists a point $q$ such that relation~$(\ref{1.1})$ holds for the point $q$ and the pseudotrajectory $\xi$. 

By the choice of $\varepsilon$, the statement of Lemma 1 holds for the constructed pseudotrajectory $\xi$ and the point $p_2$. Hence,
\begin{equation}
\label{6.1}
p\in W^s(p_2).
\tag{6.1}
\end{equation}
By similar reasons, the statement of Lemma 1 corollary holds for the pseudotrajectory $\xi$ and the point $p_1$. Hence, 
\begin{equation}
\label{6.2}
p\in W^u(p_1).
\tag{6.2}
\end{equation}
But existence of a point $p$ that satisfies both inclusions $(\ref{6.1})$ and $(\ref{6.2})$ contradicts the conditions of Case (A2). The derived contradiction means that in Case (A2) $G\notin\textrm{OSP}$.

Thus, in both possible cases we concluded that $G\notin\textrm{OSP}$. In order to finish the proof of Theorem~$\textrm{A}^{\prime}$, it remains to prove only Lemma 6.

\section{Proof of Lemma 6}

The proof of Lemma 6 is based on the proofs of certain lemmas from \cite{4}, in a great extent. 

\subsection{Item (6.c): main notations}

We start from the proof of item (6.c). By Theorem 1, there exists a hyperbolic periodic point $s$ that satisfies all conditions of item (6.c) except, perhaps, the last one: 
\begin{equation}
\label{7.1}
O(s,G)\cap(V_1\cup V_2)=\emptyset.
\tag{7.1}
\end{equation}
In fact, we repeat the major part of proof of Theorem 1 (which was formulated in Sec. 1) from the paper \cite{4}, but we need to check that, in addition to other properties, the point $s$ can be required to satisfy relation $(\ref{7.1})$. The idea of the proof is to construct the point $s$ in such a way that it would satisfy the conditions of Lemma~5. Then, relation $(\ref{7.1})$ holds, by Lemma 5. 

We assume that the sets $W_0(g_0)$ and $W_1(g_1)$ are the balls of radius $\delta$. Let us introduce the following notations:
$$\bar{f}_m[\omega]=f_{\sigma^{m-1}(\omega)}\circ\ldots\circ f_{\sigma(\omega)}\circ f_{\omega},$$
$$\bar{f}_{-m}[\omega]=f^{-1}_{\sigma^{-m}(\omega)}\circ\ldots\circ f^{-1}_{\sigma^{-1}(\omega)},$$
$$\bar{f}_0[\omega]=\textrm{id}.$$
We need the following lemma:

\begin{lemma}[Lemma on the errors]
There exists a number $K$ independent on the choice of $\delta$ such that if the inequality 
$$d_{\Sigma^2}(\omega,\omega^{\prime})\leq 2^{-m}$$
holds for a number $m\in\textbf{N}$ and points $\omega,\omega^{\prime}\in\Sigma^2$, then
$$d_{C^0}(\bar{f}_{\pm m}[\omega],\bar{f}_{\pm m}[\omega^{\prime}])\leq\gamma:=K\delta^{\beta},$$
where $\beta:=1-\frac{\ln L}{\ln 2^{\alpha}}$, and numbers $L$ and $\alpha$ were defined in the conditions of Theorem 1.   
\end{lemma}

\begin{remark}
The proof of Lemma 7 is an almost word-by-word repetition of the proof of Lemma 3.1 from the paper \cite{4}: it is necessary to make some trivial changes concerning the transition to the $C^1$-topology. We do not give it here.
\end{remark}

We need notations from the paper \cite{4}. Consider the word $\bar{\beta}=\beta_{-m}\ldots\beta_{m-1}$. By definition, put 
$$C_{\bar{\beta}}=\{\omega=\{\alpha_k\}_{k\in\textbf{Z}}\in\Sigma^2|\ \alpha_k=\beta_k\ \textrm{for}\ -m\leq k\leq m-1\}.$$
The set $C_{\bar{\beta}}$ is a cylinder neighborhood in the set $\Sigma^2$. By definition, put
$$V_{\pm}[\bar{\beta}](\phi)=\{\bar{f}_{\pm m}[\omega](\phi)|\ \omega\in C_{\bar{\beta}}\}.$$

Put $\Gamma_m=C_{\bar{\beta}}$ for a fixed word $\bar{\beta}=\beta_{-m}\ldots\beta_0\ldots\beta_{m-1}$. Define the sets $V_m(\phi)$ and $V_{-m}(\phi)$ for $\phi\in S^1$ by the relations
$$V_{\pm m}(\phi)=\{\bar{f}_{\pm m}[\omega](\phi)\ |\ \omega\in\Gamma_m\}.$$ 
Note that, by definition, $V_{\pm}[\bar{\beta}](\phi)=V_{\pm m}(\phi)$.

By Lemma 7, 
$$\textrm{diam}V_{\pm m}(\phi)\leq\gamma$$
not depending on the choice of a point $\phi$ and length of the word $\bar{\beta}$. By the definition of $\gamma$, the number $\delta$ can be chosen in such a way that $\gamma<b/40$ (the numbers $a$ and $b$ were defined in Sec. 2, when the diffeomorphisms $g_0$ and $g_1$ were being defined, and this numbers can be considered to be any sufficiently close to $1$ and $0$, respectively, but preliminary fixed numbers).

Note that there exist arches $W^+, W^-\subset S^1$, whose lengths are not less than $1/4-\delta$, such that the mapping $f_{\omega}$ expands the arch $W^+$ (with an expansion constant not less than $a-\delta$) and contracts the arch $W^-$ for any sequence $\omega$ with $\omega_0=1$, where the symbol $\omega_0$ stands at the zero position. By definition, put
$$P=\{p\in S^1| p\textrm{ is the attractor of the mapping }f_{\omega},\ \omega\in\Sigma^2, \omega_0=1\},$$
$$Q=\{q\in S^1| q\textrm{ is the repeller of the mapping }f_{\omega},\ \omega\in\Sigma^2, \omega_0=1\}.$$
Note that the values of $\textrm{diam}P$ and $\textrm{diam}Q$ are of order $\delta$, and they are not more than $\gamma=K\delta^{\beta}$ if $\delta$ is sufficiently small. Now, when the main notations are introduced, we can formulate and prove main lemmas.

\subsection{Item (6.c): main lemmas}
By definition, put $S=[1/(b-\delta)]$, where $[\cdot ]$ denotes the integer part. Note that if $\delta$ is sufficiently small, then $S$ does not depend on the choice of $\delta$.

We need the following lemma, which is a generalization of Lemma 3.3 from the paper~\cite{4}:

\begin{lemma}  
Let $\alpha=\alpha_{-n}\ldots\alpha_{n-1}$ be a word, and $\phi_1,\phi_2\in\nobreak S^1$ be two distinct points. Then there exists a word 
$$\bar{\beta}=\beta_{-m}\ldots\beta_{-n-1}\alpha_{-n}\ldots\alpha_{n-1}\beta_n\ldots\beta_{m-1}$$
such that the words $\beta_{-m}\ldots\beta_{-n-1}$ and $\beta_n\ldots\beta_{m-1}$ can not have less than $T$ zeros in a row and 
$$d_{S^1}(V_{\pm}[\bar{\beta}](\phi_1),V_{\pm}[\bar{\beta}](\phi_2))\geq 2b.$$
Hereinafter, if no additional remarks are made, we assume that distance between two sets in the circle is length of the minimal arch connecting the points of this sets. 
\end{lemma}

\begin{remark}
In general, the proof of Lemma 8 is similar with the proof of Lemma 3.3 from the paper \cite{4}. All changes in this proof are connected only with the restriction on a number of zeros in the words from the lemma. We give only an outline of the proof making an accent on necessary changes.  
\end{remark}

\begin{proof}
We construct the word $\beta^l=\beta_{-l}\ldots\beta_{l-1}$ inductively starting from the word $\alpha$ and adding by turns from one or another end of the word either $ST+1$ zeros or one unit and $ST$ zeros. We follow the algorithm described below. When the algorithm stops, our construction is completed (the reasoning is by induction on~$l$). The algorithm consists of two steps. 

Step 1. By definition, put
$$M_{\pm l}=\min_{\omega\in C_{\beta^l}}d_{S^1}(\bar{f}_{\pm l}[\omega](\phi_1),\bar{f}_{\pm l}[\omega](\phi_2)).$$
Check the following conditions:
$$M_l>3b\qquad\qquad\qquad\textrm{(B1)}\qquad\textrm{and}\qquad M_{-l}>3b.\qquad\qquad\qquad\textrm{(B2)}$$
If both conditions hold, then the algorithm stops, further we prove that in this case the constructed word satisfies the required conditions. If at least one of conditions (B1), (B2) is violated, we go to Step 2.

Step 2. By definition, put
$$W_{\pm l}=V_{\pm}[\beta^l](\phi_1)\cup V_{\pm}[\beta^l](\phi_2).$$
Check the following conditions:
$$W_l\subset W^+\qquad\qquad\qquad\textrm{(C1)}\qquad\textrm{and}\qquad W_{-l}\subset W^-.\qquad\qquad\qquad\textrm{(C2)}$$
If condition (C1) (condition (C2), respectively) holds, we add from the right (from the left, respectively) one unit and $ST$ zeros, and if it is violated, we add from the right (from the left, respectively) $ST+1$ zeros. Then, we return to Step 1 again. Denote by $\mathcal{M}$ the set of all $l$ for which we return to Step 1.  

In fact, it is proved in Proposition 3.1 from the paper \cite{4} that if the algorithm stops, then the word $\bar{\beta}$ constructed by the algorithm satisfies all conditions of Lemma 8.

In order to finish the proof of Lemma 8, it is enough to show just that the described algorithm stops after a finite number of steps. Suppose the contrary, i.e., we have constructed some growing sequence of (symmetrical) words $\beta^l$, which defines the two-sided sequence $\omega$. By definition, put
$$\phi^{\pm}_{jl}=\bar{f}_{\pm l}[\omega](\phi_j),\qquad j=1,2;$$
$$\delta^{\pm}_l=d_{S^1}(\phi^{\pm}_{1l},\phi^{\pm}_{2l}).$$
By the construction of the sequence $\omega$, $\phi^{\pm}_{jl}\in V_{\pm}[\beta^l](\phi_j)$. Hence, by Lemma 7 and by the definition of the numbers $M_{\pm l}$, 
$$\delta^{\pm}_{l}-2K\delta^{\beta}\leq M_{\pm l}\leq\delta^{\pm}_l.$$
By definition, put
$$W^{\pm}_l=V_{\pm}[\beta^l](\phi_1)\cup V_{\pm}[\beta^l](\phi_2).$$
Note that $W^{\pm}_l=W_{\pm l}$.
To continue the proof of Lemma 8, we need the following statement: 

\begin{proclaim}{Proposition 1} 
There exists a number $m\in\mathcal M$ such that $\delta^\pm_m>1/16$.
\end{proclaim}

Note that Proposition 1 implies that the described algorithm stops after a finite number of steps. Indeed, the inequality $M_{\pm m}>3b$ holds for the number $m\in\mathcal{M}$ from Proposition 1 (of course, if $b$ is sufficiently small compared to~$1/16$).

\begin{proof}[Proof of Proposition 1]
The diffeomorphism $f_{\sigma^l(\omega)}$ for $\omega_l=0$ maps any point $\phi$ to some point of the arch $[\phi+b-\delta,\phi+b+\delta]$, and maps any arch of length $\lambda$ into an arch of length $\lambda^\prime\in((1-\delta)\lambda,(1+\delta)\lambda)$. The diffeomorphism $f_{\sigma^l(\omega)}$ for $\omega_l=0$ maps any arch of length $\lambda$ that is contained in $W^+$ into an arch of length $\lambda^\prime\in((a-\delta)\lambda,(a+\delta)\lambda)$. And, if $\omega_l=1$, then $\phi^+_{1l},\phi^+_{2l}\in W^+$, by construction.

Let us show that if $\delta^+_l<1/8$ ($1/8$ is approximately one half of the arch~$W^+$), then the sequence $\omega$ can not have more than $(TS+1)(S+\nobreak1)+TS$ zeros in a row after $\omega_l$. Indeed, we apply the mappings $\delta$-close to the rotations by angles $(TS+1)b$, $(2TS+\nobreak2)b,\allowbreak\ \ldots,\ (TS+1)(S+1)b$; hence, one of this rotations maps the shortest one of the arches $(\phi^+_{1l},\phi^+_{2l})$ into the arch $W^+$ (since $W^+$ is sufficiently large). Suppose that it is the rotation by angle $(TS+1)\ell b$. If $l-1\in\mathcal{M}$, then $\omega_{l+(TS+1)\ell b}=1$, by construction. Since, it is possible that $l-1\notin\mathcal{M}$, it is necessary to take into consideration no more than $ST$ zeros required to ''get to'' the element of $\mathcal{M}$, i.e., to go to Step 1.

Hence,
\begin{equation}
\label{7.2}
\delta^+_{l+(TS+1)(S+1)+TS}>(a-\delta)(1-\delta)^{(TS+1)(S+1)+TS}\delta^+_l\qquad\textrm{for}\ \delta^+_l<1/8.
\tag{7.2}
\end{equation}
If $\delta$ is sufficiently small, then $a-\delta>1$. Since $(1-\delta)^{(TS+1)(S+1)+TS}$ is close to $1$,
$$(a-\delta)(1-\delta)^{(TS+1)(S+1)+TS}>1$$
if $\delta$ is sufficiently small.
Therefore, the sequence $\{\delta^+_n\}$ contains an element $\delta^+_n\geq1/8$. Let us prove that 
\begin{equation}
\label{7.3}
\delta^+_k>1/16\qquad\textrm{for all\ }k>n.
\tag{7.3}
\end{equation}
Suppose the contrary, i.e, there exists a number $k>n$ such that $\delta^+_k\leq 1/16$. Choose a number $t\in[n,k)$ such that
$$\delta^+_t\geq1/8,\quad\delta^+_p<1/8\qquad\textrm{for all }p\in(t,k].$$
Hence, from relation $(\ref{7.2})$ it follows that $k<t+(TS+1)(S+1)+TS+1$.
Indeed, since $\delta^+_{t+1}<1/8$, 
$$\delta^+_{t+(TS+1)(S+1)+TS+1}>(a-\delta)(1-\delta)^{(TS+1)(S+1)+ST}\delta^+_{t+1}\geq$$
$$\geq(a-\delta)(1-\delta)^{(TS+1)(S+1)+ST+1}\delta^+_t>1/8.$$     
The last inequality holds, since $\delta^+_t\geq1/8$, and number $\delta$ can be chosen so small that 
$$(a-\delta)(1-\delta)^{(TS+1)(S+1)+ST+1}>1.$$
Thus, $k<t+(TS+1)(S+1)+TS+1$. Similar reasoning shows that there can be no ones among elements of $\omega_l$ for $t+1\leq l\leq k$. But then, 
$$\delta^+_k\geq(1-\delta)^{(TS+1)(S+1)+2}1/8>1/16$$
(for a sufficiently small $\delta$), and we get a contradiction with the choice of $k$. The derived contradiction proves inequality $(\ref{7.3})$.

Similarly, it can be proved that there exists a number $n^{\prime}$ such that $\delta^-_k>1/16$ for $k>n^{\prime}$. Our reasoning implies that we can assume $k,k^{\prime}\in\mathcal{M}$. It proves Proposition 1, and, hence, Lemma 8.
\end{proof}

\end{proof}

Consider the maximal arches $\bar{W}^+$ and $\bar{W}^{-}$ such that $$N(3\gamma,\bar{W}^+)\subseteq W^+\quad\textrm{and}\quad N(3\gamma,\bar{W}^-)\subseteq W^-.$$
We assume $\gamma$ to be sufficiently small, hence, 
$$Q\subset \bar{W}^+\quad\textrm{and}\quad P\subset \bar{W}^-.$$
We need the following lemma. It is similar with Lemma 3.4 from the paper~\cite{4}: we added only one new item, item (9.a), and slightly strengthened items (9.b) and (9.c). We give only an outline of the proof of Lemma 9, emphasizing the necessary changes. In addition to it, note that above-mentioned Lemma~3.4 contains one more item, which we do not need, that is why we omit it.  

\begin{lemma}[on distortion of arches]
Suppose that we are given an arch $J\subset S^1$ and a word $\bar{\alpha}=\alpha_{-n}\ldots\alpha_0\ldots\alpha_{n-1}$. Then there exist words $$\bar{\beta}=\beta_{-m}\ldots\beta_{-n-1}\alpha_{-n}\ldots\alpha_0\ldots\alpha_{n-1}\beta_{n}\ldots\beta_{m-1}$$
and $$\bar{\beta}^\prime=\beta^\prime_{-m^\prime}\ldots\beta^\prime_{-n-1}\alpha_{-n}\ldots\alpha_0\ldots\alpha_{n-1}\beta^\prime_{n}\ldots\beta^\prime_{m^\prime-1}$$ 
such that 
\begin{enumerate}
\item[(9.a)] the words that were added to the word $\bar{\alpha}$ can not contain less than $T$ zeros in a row;
\item[(9.b)] if $\omega\in C_{\bar{\beta}}$, then 
\begin{equation}
	\label{7.4} 
	\bar{f}_m[\omega](J)\subset \bar{W}^-\quad\textrm{and}\quad W^-\subset\bar{f}_{-m}[\omega](J),
  \tag{7.4}
	\end{equation}
	\begin{equation}
	\label{7.5}
	|(\bar{f}_m[\omega])^{\prime}|_{J}|<1,\qquad |(\bar{f}_{-m}[\omega])^{\prime}|_{(\bar{f}_{-m}[\omega])^{-1}(W^-)}|>1;
  \tag{7.5}
	\end{equation}
\item[(9.c)] if $\omega^\prime\in C_{\bar{\beta^\prime}}$, then 
\begin{equation}
	\label{7.6}
	\bar{f}_{-m^\prime}[\omega^\prime](J)\subset \bar{W}^+\quad\textrm{and}\quad W^+\subset\bar{f}_{m^\prime}[\omega^\prime](J),
  \tag{7.6}
	\end{equation}
	\begin{equation}
	\label{7.7}
	|(\bar{f}_{-m^\prime}[\omega^\prime])^{\prime}|_{J}|<1,\qquad |(\bar{f}_{m^\prime}[\omega^\prime])^{\prime}|_{(\bar{f}_{m^\prime}[\omega^\prime])^{-1}(W^+)}|>1.
	\tag{7.7} 
	\end{equation}
\end{enumerate}  
\end{lemma}

\begin{proof}
Construct the word $\bar{\beta}$ with properties (9.a) and (9.b). Denote by $\phi_1$ and $\phi_2$ the ends of $J$. By Lemma 8, the word $\bar{\alpha}$ can be transformed into a word $\bar{\beta_1}=\beta_{-k_1}\ldots\beta_{k_1-1}$ in such a way that it would satisfy the statement of Lemma 8, i.e., the distances between the sets $V_{k_1}(\phi_1)$ and  $V_{k_1}(\phi_2)$, and between the sets $V_{-k_1}(\phi_1)$ and $V_{-k_2}(\phi_2)$ are not less than $2b$.

By definition, put 
$$X_{-l}:=\bigcap_{\omega\in\Gamma_l}\bar{f}_{-l}[\omega](J)\quad\textrm{for }l\geq k_1,$$
$$Y_{l}:=\bigcap_{\omega\in\Gamma_l}\bar{f}_{l}[\omega](S^1-\bar{J})\quad\textrm{for }l\geq k_1;$$
i.e., $X_{-l}$ is the interval between the sets $V_{-l}(\phi_1)$ and $V_{-l}(\phi_2)$ contained in the images of the arch $J$ by the mapping $\bar{f}_{-l}[\omega]$, where $\omega$ ''runs through'' $\Gamma_l$; and $Y_l$ is the interval between $V_l(\phi_1)$ and $V_l(\phi_2)$ contained in the images of the arch $S^1-\bar{J}$ by the diffeomorphism $\bar{f}_l[\omega]$, where $\omega$ ''runs through'' the set $\Gamma_l$ (we denote by $\bar{J}$ the closure of the arch $J$ and by the symbol ''$-$'' the set difference).

We transform the word $\bar{\beta}_{1}$  into the word $\bar{\beta}_2=\beta_{-k_2}\ldots\beta_{k_2-1}$ in such a way that it satisfies item (9.a) and inclusions 
$$Q\subset Y_{k_2}\qquad\textrm{and}\qquad P\subset X_{-k_2}.$$
For this purpose, we add by induction (like in Lemma 8) symbols both from the left and from the right. Let $l$ be the induction parameter, the case $l=k_1$ is the induction base.

Check two following conditions:
$$Q\subset Y_l\qquad (\textrm{D}1)\qquad\textrm{and}\qquad P\subset X_{-l}\qquad (\textrm{D}2).$$
If both conditions hold, the construction is completed. Otherwise, we do the following: if condition (D1) (condition (D2), respectively) holds, then we add from the right (from the left) $ST+1$ ones, otherwise, we add from the right (from the left) $ST+1$ zeros.  

Let us show that this algorithm stops after a finite number of iterations.

\begin{proclaim}{Proposition 2} If condition (D1) or (D2) holds on some iteration, then it will be satisfied up to the end of the construction.
\end{proclaim}
 
The essence of this proposition can be formulated as follows: addition of any number of ones can not ''hurt'' this conditions. The proof of Proposition~2 (as well as its formulation) is a word-by-word repetition of the proof of Proposition~3.4 from the paper \cite{4}, that is why we omit it. 

\begin{proclaim}{Proposition 3}
Each of the conditions (D1) and (D2) holds on some iteration.
\end{proclaim}

\begin{proof}[Proof of Proposition 3]
Suppose that condition (D1) never holds; the case of condition (D2) can be treated in the same way. The algorithm described above defines a certain sequence $\omega$, and, by Proposition 2, $\omega_l=0$ for $l\geq k_1$. Hence, the mappings $f_{\sigma^l(\omega)}$ are close to the rotation by angle $b$ for $l\geq k_1$. In addition, by the choice of $k_1$, the distance between the sets $V_{k_1}(\phi_1)$ and $V_{k_1}(\phi_2)$ is not less than $2b$. Hence, 
$$\textrm{diam}(\bar{f}_{k_1}[\omega](S^1-\bar{J}))\geq 2b.$$
But then, if $\delta$ is sufficiently small, one of the arches 
$$\bar{f}_{k_1}[\omega](S^1-\bar{J}),\ \bar{f}_{k_1+ST+1}[\omega](S^1-\bar{J}),\ldots,\bar{f}_{k_1+(S+1)ST+S+1}[\omega](S^1-\bar{J})$$ 
covers the set $Q$ in such a way that the distance from $Q$ to its ends is more than $\gamma$. Let it be the arch $\bar{f}_{k_1+tST+t}[\omega](S^1-\bar{J})$. By Lemma 7, the sets $V_{k_1+tST+t}(\phi_1)$ and $V_{k_1+tST+t}(\phi_2)$ are contained in the $\gamma$-neighborhoods of the ends of this arch. That is why $Q\subset Y_{k_1+tST+t}$, and condition (D1) holds.
\end{proof}

Thus, we have constructed the word $\bar{\beta}_2$ of length $2k_2$ that satisfies the analogs of condition (9.a) from the formulation of the lemma and conditions (D1) and (D2). To finish the proof, one should repeat the reasoning from the end of the proof of Lemma 3.4 from the paper \cite{4} with minor changes. Lemma~9 is proved.
\end{proof}

Now, when Lemma 9 is proved, we can finish the proof of item (6.c).

\subsection{Item (6.c): end of the proof}

Note that sets of form $C_{\alpha}\times J\subset\Sigma^2\times S^1$, where $J\subset S^1$ is an arch and $\alpha$ is a word, form a base of topology in the set $\Sigma^2\times S^1$. Suppose that, as above, $\omega_3$ is a periodically with period $T_3$ repeating word in the sequence $pr p_3$. Choose so large number $2m$ and so small arch $J$ that
$$p_3\subset C_{\omega_3\ldots\omega_3}\times J\subset N(d,p_3),$$
where the word $\omega_3$ is repeated precisely $m$ times before the zero position and precisely $m$ times after. Next, we apply item (9.c) from Lemma 9. Suppose that $\bar{\beta}$ is the word from item (9.c) of Lemma 9. By $\omega$ denote an infinite sequence in which the word $\bar{\beta}$ is repeated periodically (and $\omega\in C_{\bar{\beta}}$). For any point $\phi\in S^1$ Lemma 5 can be applied to the point $(\omega,\phi)\in\Sigma^2\times S^1$ (since property (9.a) holds). Hence, condition $(\ref{7.1})$ holds for any point $s=(\omega,\phi)$. 

The paper \cite{4} proves that there exists a point $\phi_0\in S^1$ such that the point $s:=(\omega,\phi_0)$ is a hyperbolic periodic point of type $(1,2)$, and conditions $(\ref{7.6})$ and $(\ref{7.7})$ hold. Item (6.c) is proved.

\subsection{Proof of the remaining items of Lemma 6}

We give only the proof of item (6.d). Note that the points $p_1$, $p_2$, $p_3$ and $p_4$ are preserved for the considered mild skew products $G$, and there was a lot of freedom in the choice of this points (indeed, only hyperbolicity and condition $(\ref{5.1})$ on periods were required). In Subsec. 7.1--7.3 there was given a sufficiently detailed description of described in~\cite{4} procedure for construction of hyperbolic periodic points $p=(\omega,\phi)$ of different types that satisfy conditions $(\ref{7.4})$ and $(\ref{7.5})$, or $(\ref{7.6})$ and $(\ref{7.7})$, respectively, (depending on the type of the periodic point). This procedure allows to construct points of arbitrary large periods. That is why it can be assumed that the points $p_1$ and $p_4$ were initially constructed by such procedure for the step skew product $G_0$ and then fixed. In this case, items (6.a) and (6.b) are consequences of item (6.d). Thus, it is enough to prove only item (6.d).

Let us give a brief outline of the proof scheme of item (6.d). At first, we construct a pseudotrajectory $\omega\in\Sigma^2$ such that 
$$\omega\in W^u(pr s)\cap W^s(pr p_4)\quad\textrm{and}\quad O(\omega, \sigma)\cap (pr V_1\cup pr V_2)=\emptyset;$$
the sequence $\omega$ ''includes'' a certain subsequence of the sequence $pr s$ before the zero position, ''includes'' zeros from the zero position to the $(\bar{K}-1)$-position and ''includes'' a certain subsequence of the sequence $pr p_4$ after the $\bar{K}$-position (the number $\bar{K}$ is an arbitrary number at the moment, further it will be chosen implicitly). Then, we prove that for the constructed sequence $\omega$ there exists a point $\phi\in S^1$ such that the trajectory of the point $(\omega,\phi)$ ''goes'' from the point $s$ to the point $p_4$. Next, we apply Proposition 5 (an analog of Lemma5) and see that the trajectory of the point $y:=(\omega,\phi)$  does not intersect the cylinder neighborhoods $V_1$ and $V_2$ of the sets $O(p_1,G)$ and $O(p_2,G)$. Thus, the statement of item (6.d) holds for the point $y$. 

Choose arbitrary numbers $\bar{K}\in\textbf{Z}$ and $m\in\textbf{N}$. We say that two sequences $\omega=\{\beta_k\}_{k\in\textbf{Z}}$ and $\bar{\omega}=\{\bar{\beta}_k\}_{k\in\textbf{Z}}$ coincide on the interval $[\bar{K}-m,\bar{K}+m-1]$ if the relation
\begin{equation}
\label{7.8}
\beta_k=\bar{\beta}_k\quad\textrm{for }\bar{K}-m\leq k\leq \bar{K}+m-1
\tag{7.8}
\end{equation}
holds. We need the following statement that is, formally, a generalization of Lemma 7 on errors. Its proof is trivial.

\begin{proclaim}{Proposition 4}
If the sequences $\omega,\bar{\omega}\in\Sigma^2$ satisfy relation $(\ref{7.8})$, then the inequality 
\begin{equation}
\label{7.9}
d_{S^1}(\bar{f}_{\bar{K}\pm m}[\omega](\bar{\phi}_1),\bar{f}_{\bar{K}\pm m}[\bar{\omega}](\bar{\phi}_2))\leq\gamma
\tag{7.9}
\end{equation}
holds, where $\gamma$ is the constant from Lemma 7, and points $\bar{\phi}_1,\bar{\phi}_2\nobreak\in\nobreak S^1$ are defined by the equalities
\begin{equation}
\label{7.10}
\bar{\phi}_1:=(\bar{f}_{\bar{K}}[\omega])^{-1}(\phi)\in S^1,\quad\bar{\phi}_2:=(\bar{f}_{\bar{K}}[\bar{\omega}])^{-1}(\phi)\in S^1.
\tag{7.10}
\end{equation}
\end{proclaim}

Recall that the number $T$ is defined by equality $(\ref{5.2})$. Denote by $t_p$ the period of the point $p_4$ and by $t_s$ the period of the point $s$. Note that $t_s>T$. We can assume that $p_4=(\alpha^p,\phi_p)$, and a word $\alpha^p_1\ldots\alpha^p_{t_p}$ is repeated periodically in the sequence $\alpha^p$ in such way that the symbol $\alpha_1^p$ stands at the zero position of the sequence~$\alpha^p$. Recall that the points $s$ and $p_4$ are repellers on the fibres. In Subsec.~6.3, we chose the maximal arch $W^+$ such that
$$N(3\gamma,\bar{W}^+)\subset W^+.$$
Since the point $p_4=(\alpha^p,\phi_p)$ is periodic,
$$\bar{f}_{t_p}[\alpha^p](\phi_p)=\phi_p.$$
From the construction of the point $s$ (by item (9.c) from Lemma 9) it follows that the analogs of relations $(\ref{7.6})$ and $(\ref{7.7})$ hold for the sequence $\omega^s=pr s$, the arch $J$ and the certain number $m_s$ (defined in Lemma 9). Hence, if $\bar{s}=(\alpha^s,\phi_s)=G^{m_{s}}(s)$, and a word $\alpha^s_1\ldots\alpha^s_{t_s}$ is repeated periodically in the sequence $\alpha^s$ in such way that the symbol $\alpha^s_{t_s}$ stands at the $(-1)$-position of the sequence $\alpha^s$, then the relations 
$$\phi_s\subset \bar{W}^+\subset W^+,\quad \bar{f}_{-t_s}[\alpha^s](W^+)\subset \bar{W}^+\subset W^+,\quad
|(\bar{f}_{-t_s}[\alpha^s])^{\prime}|_{W^+}|<1$$
hold. This relations mean that the arch $(\alpha^s,W^+)$ is contained in the repelling domain of the point $\bar{s}$ with respect to the mapping $G^{t_s}$, i.e., the repelling domain on fibres of the point $(\alpha^s,\phi_s)$ (the repelling domain of the point $\phi_s$ for the restriction of the mapping $G^{t_s}$ on the set $(\alpha^s,S^1)$))  contains the arch~$W^+$.

Since the point $\bar{s}=(\alpha^s,\phi_s)$ is periodic, 
$$\bar{f}_{-t_s}[\alpha^s](\phi_s)=\phi_s.$$
Choose a set $\Delta_p$ which is a neighborhood of the point $p_4$ such that if 
$$O_{+}(p,G^{t_p})\subset\Delta_p$$
for some point $p$, then $p\in W^s(p_4)$.

Assume that the sequence $\omega$ is such that
\begin{enumerate} 
\item it includes a word $\alpha^s_1\ldots\alpha^s_{t_s}$ from the $(-t_s)$-position to the $(-1)$-position, and this word is further periodically repeated in the subsequence $\omega_{k<0}$ (the sequence $\omega$ can be considered as the mapping $\omega:\textbf{Z}\mapsto M$, then $\omega|_{A}$ is the restriction of the mapping $\omega$ to a set $A$);
\item it includes a word $\alpha^p_1\ldots\alpha^p_{t_p}$ from the $\bar{K}$-position to the $(\bar{K}+t_p-1)$-position, and this word is further periodically repeated in the subsequence $\omega|_{k\geq \bar{K}}$, where $\bar{K}\in\textbf{N}$ is a certain number, which will be chosen later.   
\end{enumerate}
    
\begin{lemma}
\begin{enumerate}
\item[(10.a)] For any number $m\in\textbf{N}$ there exist a point $\phi^m_{\omega}$ and an arch $J^m_{\omega}$ such that 
\begin{equation}
	\label{7.11}
	\phi^m_{\omega}\in N(\gamma,\phi_p);
	\tag{7.11}
	\end{equation}
	\begin{equation}
	\label{7.12}
	\bar{W}^+\subset N(\gamma,J^m_{\omega})\quad\textrm{and}\quad J^m_{\omega}\subset N(\gamma,\bar{W}^+)\subset W^+;
  \tag{7.12}
  \end{equation}	
\begin{equation}
\label{7.13}
d_{S^1}(\bar{f}_{wt_p}[\sigma^{\bar{K}}(\omega)](\phi^m_{\omega}),\phi_p)\leq\gamma\quad\textrm{for }0\leq w\leq 2m;
\tag{7.13}
\end{equation}
	\begin{equation}
	\label{7.14}
	\bar{f}_{-wt_s}[\omega](J^m_{\omega})\subset N(\gamma,\bar{f}_{-wt_s}[\alpha^s](\bar{W}^+))\subset W^+\quad\textrm{for }0\leq w\leq 2m.
  \tag{7.14}
  \end{equation}
\item[(10.b)] If $\phi_{\omega}$ is one of the limit points of the sequence $\phi^m_{\omega}$, and an arch $J_{\omega}$ is a ''limit arch'' of the sequence $J^m_{\omega}$ (the meaning of this term will be clarified in the proof of the lemma), then the relations
\begin{equation}
\label{7.15}
(\sigma^{\bar{K}}(\omega),\phi_{\omega})\in W^s(p_4),\quad (\omega, J_{\omega})\subset W^u(\bar{s})
\tag{7.15}
\end{equation}
hold.
\end{enumerate}
\end{lemma}

\begin{proof}
We start from the proof of item (10.a). Choose an arbitrary number $m\in\textbf{N}$ and an arbitrary number $0\leq k< m$.

Put $\bar{L}=mt_p$. Consider the sequence $\alpha^p$. Next, we apply Proposition 4 to the ''interval'' $[\bar{L}-(m-k)t_p,\bar{L}+(m-k)t_p-1]=[kt_p,2mt_p-kt_p-1]$ and sequences $\sigma^{\bar{K}}(\omega)$ and $\alpha^p$, which coincide on this interval, by construction. Put $\phi^m_{\omega}=(\bar{f}_{mt_p}[\sigma^{\bar{K}}(\omega)])^{-1}(\phi_p)$. By construction,
$$\bar{f}_{t_p}[\alpha^p](\phi_p)=\phi_p.$$
Hence, by inequalities $(\ref{7.9})$ and $(\ref{7.10})$, 
\begin{equation}
\label{7.16}
d_{S^1}(\bar{f}_{kt_p}[\sigma^{\bar{K}}(\omega)](\phi^m_{\omega}),\bar{f}_{kt_p}[\alpha^p](\phi_p))\leq\gamma,
\tag{7.16}
\end{equation}
\begin{equation}
\label{7.17}
d_{S^1}(\bar{f}_{2mt_p-kt_p}[\sigma^{\bar{K}}(\omega)](\phi^m_{\omega}),\bar{f}_{2mt_p-kt_p}[\alpha^p](\phi_p))\leq\gamma
\tag{7.17}
\end{equation}
for any $0\leq k< m$. Next, we set $k=0$ in inequality $(\ref{7.16})$ and get inclusion $(\ref{7.11})$. Inequalities $(\ref{7.13})$ for $w\neq m$ follow from inequalities $(\ref{7.16})$ and $(\ref{7.17})$. Inequalities $(\ref{7.13})$ for $w=m$ hold, by construction. 

Put $\bar{L}=-mt_s$. Next, we apply Proposition 4 to the ''interval''
$$[\bar{L}-(m-k)t_s,\bar{L}+(m-k)t_s-1]=[-2mt_s+kt_s,-kt_s-1]$$
and the sequences $\omega$ and $\alpha^s$, which coincide on this interval, by construction. Put $V^+_k=\bar{f}_{-kt_s}[\alpha^s](\bar{W}^+)$ and $J^m_{\omega}=(\bar{f}_{-mt_s}[\omega])^{-1}(V^+_m)$. It is clear that the set $J^m_{\omega}$ is an arch. By construction, the set $(\alpha^s,W^+)$ is contained in the repelling domain of the point $\bar{s}$, and, moreover, 
$$\bar{f}_{vt_s}[\alpha^s](W^+)\subset \bar{W}^+\subset W^+\quad\textrm{for } v\in\textbf{Z},v\leq0.$$
By construction of the arch $\bar{W}^+$,
$$V^+_{-v}=\bar{f}_{vt_s}[\alpha^s](\bar{W}^+)\subset \bar{W}^+\quad\textrm{for } v\in\textbf{Z},v\leq0.$$

Hence, by relations $(\ref{7.9})$ and $(\ref{7.10})$,
\begin{equation}
\label{7.18}
d_{H}(\bar{f}_{-2mt_s+kt_s}[\omega](J^m_{\omega}),V^+_{2m-k})\leq\gamma,
\tag{7.18}
\end{equation}
\begin{equation}
\label{7.19}
d_{H}(\bar{f}_{-kt_s}[\omega](J^m_{\omega}),V^+_k)\leq\gamma
\tag{7.19}
\end{equation}
for all $0\leq k< m$, where $d_H$ denotes the Hausdorf distance.

Next, we set $k=0$ in the inequality $(\ref{7.19})$ and get inclusion $(\ref{7.12})$. Inequalities $(\ref{7.14})$ for $w\neq m$ follow from inequalities $(\ref{7.18})$ and $(\ref{7.19})$. Inequalities $(\ref{7.14})$ for $w=m$ hold, by construction.

Let us prove item (10.b). Let $\phi_{\omega}$ be a limit point of the sequence $\phi_{\omega}^m$. Then, relations $(\ref{7.13})$ and inclusion $(\ref{7.11})$ hold for the point $\phi_{\omega}$ and an arbitrary number $w$. Let $j_1^m$~and~$j_2^m$ be the ends of the arch $J^m_{\omega}\subset \bar{W}^+$. Then, there exists a sequence $m_k$ such that $$j_e^{m_k}\longrightarrow j_e\qquad\textrm{for }e=\{0,1\},\ k\rightarrow+\infty,$$
where $j_1$ and $j_2$ are some points. Let $J_{\omega}$ be the arch between the points $j_1$ and $j_2$ that is contained in $W^+$. The arch $J_{\omega}$ is called the limit arch. Note that $J_{\omega}$ is the set of all limit points of the sequence $J^{m_k}_{\omega}$. The arch $J_{\omega}$ satisfies relations $(\ref{7.14})$ and inclusions $(\ref{7.12})$ for arbitrary $w$.

Put 
$$\bar{\Delta}_p=\{p_4+(x-p_4)/2|x\in\Delta_p\}.$$
Since the points $p_1$, $p_2$, $p_3$ and $p_4$ were ''fixed'' for all mild skew products from Theorem~$\textrm{A}^{\prime}$, we can assume that $\delta$ was chosen to be so small that 
$$N(2\gamma,\bar{\Delta}_p)\subset\Delta_p.$$
By relation $(\ref{7.13})$, the positive semitrajectory of the point $(\sigma^{\bar{K}}(\omega),\phi_{\omega})$ with respect to the mapping $G^{t_p}$ is contained in the neighborhood $\bar{\Delta}_p$ (for sufficiently small $\gamma$; recall that the point $p_4$ is fixed, and the point $\bar{s}$ is not fixed). Consequently,
$$(\sigma^{\bar{K}}(\omega),\phi_{\omega})\in W_{G^{t_p}}^s(p_4);\quad\textrm{hence,}\quad(\sigma^{\bar{K}}(\omega),\phi_{\omega})\in W^s(p_4).$$
Similarly, by relations $(\ref{7.14})$, the negative semitrajectories of the points of the arch $(\omega,J_{\omega})$ belong to a small neighborhood of the arch $(pr\bar{s},\bar{W}^+)$.

It was already noted above that the repelling domain of the point $\bar{s}$ on the fibres with respect to the mapping $G^{t_s}$ contains the arch $W^+$, and that is why the repelling domain of the point $\bar{s}$ with respect to the mapping $G^{t_s}$ is sufficiently large, to be precise, it contains a subset of the form
$$(V(pr\bar{s})\cap W_{\sigma}^u(pr\bar{s}),N(\gamma,\bar{W}^+)),$$
where $V(pr\bar{s})$ is a small neighborhood of the point $pr\bar{s}$ in the base, i.e., in the set $\Sigma^2$.

That is why relations $(\ref{7.14})$ imply the inclusion
$$(\omega,J_{\omega})\subset W_{G^{t_s}}^u(\bar{s}),\quad\textrm{hence,}\quad
(\omega,J_{\omega})\subset W^u(\bar{s}).$$ 
Lemma 10 is proved. 
\end{proof}

In particular, Lemma 10 means that relations $(\ref{7.15})$ hold for a certain point $\phi_{\omega}$ and a certain arch $J_{\omega}$. Let us define the symbols that are contained in the ''interval'' from the zero position to the $(\bar{K}-1)$-position in the sequence $\omega$. Recall that in Subsec. 7.2 we introduced the number $S=[1/(b-\delta)]$, where number $b$ is such that the diffeomorphism $g_0$ defined above is the rotation by the angle $b$. By inclusions~$(\ref{7.12})$, the arch $J_{\omega}$ is sufficiently large, consequently, one of the arches 
$$\bar{f}_{TS+1}[\omega](J_{\omega}), \ldots, \bar{f}_{kTS+k}[\omega](J_{\omega}), \ldots, \bar{f}_{(S+1)TS+S+1}[\omega](J_{\omega})$$
''covers'' the point $\phi_p$. Suppose that it happened for the arch $\bar{f}_{kTS+k}[\omega](J_{\omega})$. In this case, by definition, we put
$$\bar{K}:=kTS+k,$$
and define all symbols of the sequence $\omega$ in the interval from the zero position to the $(\bar{K}-1)$-position to be equal to zero.

By Lemma 10, $(\omega,J_{\omega})\subset W^u(\bar{s})$ and $(\sigma^{\bar{K}}(\omega),\phi_p)\in W^s(p_4)$; by construction,      
$(\sigma^{\bar{K}}(\omega),\phi_p)\in G^{\bar{K}}(\omega,J_{\omega})$;
consequently, 
$$(\sigma^{\bar{K}}(\omega),\phi_p)\in W^u(\bar{s})\cap W^s(p_4).$$
In order to finish the proof of Lemma 6, we need to prove only the following statement. The proof of this statement is similar with the proof of Lemma 5, that is why we omit it.

\begin{proclaim}{Proposition 5}
Suppose that a word $\beta_p$ is repeated periodically with period $T_4$ in the sequence $pr p_4$, $\beta_s=\alpha^s_1\ldots\alpha^s_{t_s}$ is the word constructed above which is periodically repeated in the sequence $pr s$, and $\theta$ is a word that consists of $k(TS+1)$ zeros for $k\geq 0$. If $\omega=\ldots\beta_s\ldots\beta_s\theta\beta_p\ldots\beta_p\ldots$,
then 
$$O(\omega,\sigma)\cap (pr V_1\cup pr V_2)=\emptyset.$$      
\end{proclaim}

Next, we apply Proposition 5 to the point $y:=(\sigma^{\bar{K}}(\omega),\phi_p)$ and see that the point $y$ satisfies all conditions of item (6.d). Lemma 6 is proved. Hence, Theorems A and $\textrm{A}^{\prime}$ are proved too.

\end{document}